\documentclass[a4paper,11pt]{amsart}
\usepackage{amssymb, amsmath}
\usepackage{amscd}
\numberwithin{equation}{section}
\usepackage{epsfig}
\usepackage{amsmath}
\usepackage{amsfonts,amssymb,amsopn}

\usepackage{hyperref}
\usepackage{mathrsfs}

\usepackage{verbatim}
\usepackage[english]{babel}

\usepackage{amsthm}
\usepackage{hyperref}
\usepackage{color}
\usepackage[color,all]{xy}

\newtheorem{theorem}{{\textbf Theorem}}[section]
\newtheorem{proposition}[theorem]{{\textbf Proposition}}
\newtheorem{corollary}[theorem]{{\textbf Corollary}}
\newtheorem{lemma}[theorem]{{\textbf Lemma}}
\newtheorem{defn}[theorem]{{\textbf Definition}}
\newtheorem{remit}[theorem]{{\textbf Remark}}
\newenvironment{remark}{\begin{remit}\rm}{\end{remit}}
\newenvironment{definition}{\begin{defn}\rm}{\end{defn}}
\newcommand{\cB}{{\mathcal B}}
\newcommand{\cC}{{\mathcal C}}
\newcommand{\cD}{{\mathcal D}}
\newcommand{\cE}{{\mathcal E}}
\newcommand{\cF}{{\mathcal F}}
\newcommand{\cH}{{\mathcal H}}
\newcommand{\cI}{{\mathcal I}}
\newcommand{\cM}{{\mathcal M}}
\newcommand{\cN}{{\mathcal N}}
\newcommand{\cO}{{\mathcal O}}
\newcommand{\cP}{{\mathcal P}}
\newcommand{\cQ}{{\mathcal Q}}
\newcommand{\cR}{{\mathcal R}}
\newcommand{\cT}{{\mathcal T}}
\newcommand{\cV}{{\mathcal V}}
\newcommand{\bC}{{\mathbb C}}
\newcommand{\pp}{{\mathbb P}}

\newcommand{\bX}{{\mathbb X}}
\newcommand{\bY}{{\mathbb Y}}
\newcommand{\Z}{{\mathbb Z}}

\newcommand{\scrC}{\mathscr{C}}
\newcommand{\tscrC}{\widetilde{\scrC}}
\newcommand{\scrE}{\mathscr{E}}
\newcommand{\scrL}{\mathscr{L}}

\newcommand{\scrV}{\mathscr{V}}

\newcommand{\bfE}{\mathbf{E}}

\newcommand{\tV}{\widetilde{V}}

\newcommand{\bE}{\overline{E}}

\newcommand{\be}{\begin{equation}}
\newcommand{\ee}{\end{equation}}

\newcommand{\bfmu}{\mathbf{\mu}}

\newcommand{\Qt}{\widetilde{Q}}
\newcommand{\Pt}{\widetilde{P}}

\newcommand{\Sym}{\mathrm{Sym}}
\newcommand{\Image}{\mathrm{Im}\,}
\newcommand{\Iden}{\mathrm{Id}}
\newcommand{\rank}{\mathrm{rk}\,}
\newcommand{\Pic}{\mathrm{Pic}}
\newcommand{\ev}{\mathrm{ev}}
\newcommand{\sm}{\mathrm{sm}}

\newcommand{\Oc}{{{\cO}_{C}}}

\newcommand{\Gr}{\mathrm{Gr}}

\newcommand{\OG}{\mathrm{OG}}
\newcommand{\Quot}{\mathrm{Quot}}
\newcommand{\cQuot}{\cQ uot}
\newcommand{\Supp}{\mathrm{Supp}}
\newcommand{\Sing}{\mathrm{Sing}}
\newcommand{\isom}{\xrightarrow{\sim}}

\newcommand{\codim}{\mathrm{codim}}
\newcommand{\length}{\mathrm{length}}

\newcommand{\GL}{\mathrm{GL}}
\newcommand{\SO}{\mathrm{SO}}
\newcommand{\GO}{\mathrm{GO}}
\newcommand{\Orth}{\mathrm{O}}
\newcommand{\Spin}{\mathrm{Spin}}

\newcommand{\MO}{{\cM\cO}_C (2n ; \delta)}
\newcommand{\IQ}{\mathrm{IQ}}
\newcommand{\IQo}{\IQ^\circ}
\newcommand{\IQe}{\IQ_e (V)}
\newcommand{\IQeo}{\IQ^\circ_e (V)}
\newcommand{\IQeV}{\IQ_e(V)}
\newcommand{\IQeoV}{\IQ_e^\circ(V)}

\newcommand{\IQoer}{\IQo_{e + 2r} ( V )_\delta}

\newcommand{\dpr}{{\delta^\prime}}

\newcommand{\tPi}{\widetilde{\Pi}}

\def\cOB{{\cOB}}

\def\bIQo{\mathbb{IQ}^\circ}
\def\bIQ{\mathbb{IQ}}

\newcommand{\bIQeV}{\bIQ_e(V)}
\newcommand{\bIQeoV}{\bIQ_e^\circ(V)}
\newcommand{\bIQetV}{\bIQ_e(\tV)}

\newcommand{\bIQe}{\bIQ_e(V)}
\newcommand{\bIQeW}{\bIQ_e(W)}

\newcommand{\tpi}{\widetilde{\pi}}
\newcommand{\tB}{\widetilde{B}}

\newcommand{\tU}{\widetilde{U}}
\newcommand{\tN}{\widetilde{N}}

\title[Counting maximal isotropic subbundles]{Counting maximal isotropic subbundles of orthogonal bundles over a curve}

\author{Daewoong Cheong}
\address{Chungbuk National University, Department of Mathematics, Chungdae-ro 1, Seowon-Gu, Cheongju City, Chungbuk 28644, Korea}
\email{daewoongc@chungbuk.ac.kr}

\author{Insong Choe}
\address{Department of Mathematics, Konkuk University, 1 Hwayang-dong, Gwangjin-Gu, Seoul 143-701, Korea}
\email{ischoe@konkuk.ac.kr}

\author{George H.\ Hitching}
\address{Oslo Metropolitan University, Postboks 4, St. Olavs plass, 0130 Oslo, Norway}
\email{gehahi@oslomet.no}



\begin{document}

\begin{abstract}
Let $C$ be a smooth projective curve and $V$ an orthogonal bundle over $C$. Let $\IQeV$ be the isotropic Quot scheme parameterizing degree $e$ isotropic subsheaves of maximal rank in $V$. We give a closed formula for intersection numbers on components of $\IQeV$ whose generic element is saturated. As a special case, for $g \ge 2$, we compute the number of isotropic subbundles of maximal rank and degree of a general stable orthogonal bundle in most cases when this is finite. This is an orthogonal analogue of Holla's enumeration of maximal subbundles in \cite{Ho}, and of the symplectic case studied in \cite{CCH1}.
\end{abstract}

\maketitle

\section{Introduction}

Let $V$ be a vector bundle of rank $r$ over a smooth projective curve $C$ of genus $g \ge 2$. For $1\leq k \leq r-1$, a rank $k$ subbundle $E$ of $V$ is called a \textsl{maximal subbundle} if $E$ has maximal degree among rank $k$ subbundles of $V$. When a bundle $V$ has finitely many rank $k$ maximal subbundles,  the  problem of counting these subbundles has been the subject of much attention. For example, Segre \cite{Seg} and Nagata \cite{Na} proved that if $d \not\equiv g \mod 2$, then a general stable bundle of rank two has $2^g$ maximal line subbundles. Results for higher rank cases appeared in \cite{LN}, \cite{Ox},  \cite{Te} and Holla \cite{Ho} gave a closed formula for the number of maximal subbundles, using Gromov--Witten invariants of the Grassmannian. An analogous formula was obtained  in \cite{CCH1} for Lagrangian subbundles of a symplectic bundle.

The objective of the present work is to give a similar result for orthogonal bundles. For a line bundle $L$ over $C$, an \textsl{$L$-valued orthogonal bundle} over $C$ is a vector bundle $V$ equipped with a nondegenerate symmetric bilinear form $\omega \colon V \otimes V \to L$. A subsheaf $E$ of $V$ is called \textsl{isotropic} if $\omega_{E \otimes E} = 0$. By linear algebra, the rank of an isotropic subsheaf is at most $\frac{1}{2}\rank \, V$. By \cite{CCH2}, an $L$-valued  orthogonal bundle $V$ of  rank $2n$ has an isotropic subbundle of  rank $n$ if and only if $\det V \cong L^n$. On the other hand, any orthogonal bundle of  rank $2n+1$ has an isotropic subbundle of rank $n$.
 In this paper,  for most configurations of numerical invariants, we compute the number of isotropic subbundles  of rank $n$ and maximal degree when this is finite.

 Let $V$ be an $L$-valued orthogonal bundle of rank $r$ and $\IQ_e(V)$  the subscheme of the  Quot scheme of $V$ which parameterizes isotropic subsheaves of $V$ of rank $n:= \left\lfloor \frac{r}{2} \right\rfloor$ and degree $e$.
By Riemann--Roch, the expected dimension of $\IQ_e(V)$ is given by $I( n, \ell, e)$  if $r=2n$ and $I( n+1, \ell, e +\frac{\ell}{2})$ if $r=2n+1$, where $\ell = \deg (L)$ and
 \begin{equation}
 I ( n, \ell, e ) \ := \ -( n-1) e - \frac{1}{2} n ( n - 1 ) ( g-1-\ell   ) . \label{Inle}
 \end{equation}

Now  suppose there is an integer $e_0$  such that  that the expected dimension of $\IQ_{e_0}(V)$ is  zero.
We show for a general orthogonal bundle $V$ that  $\IQ_{e_0}(V)$, if nonempty,   is a  smooth scheme of dimension zero (Theorem \ref{e_0maximal}).
Hence  in this case the length of the scheme $\IQ_{e_0}(V)$ is equal to the number  of isotropic subbundles of $V$ of rank $n$ and degree $e_0$, which will be denoted by $N( g , r , \ell , e_0 )$.
The goal of this paper is to provide a formula to compute these numbers.

 Our main results are Corollary  \ref{counting formula} for the even rank case and Proposition \ref{counting odd rank} for the odd rank case, which give formulas to compute $N(g,r,\ell, e_0)$. These formulas cover almost all cases with small exceptions; see Remark \ref{Remark:Not have}.
Explicit counting results for  low rank cases are stated in \S\:\ref{low_rank}.

Let us briefly explain the strategy for obtaining the formula. For the time being, consider  the even rank case $r = 2n$; the result for the odd rank case will follow from the even rank case.  In contrast to the classical and symplectic cases, in general the scheme $\IQe$ has multiple components, most of which parameterize nonsaturated subsheaves. To 
 circumvent this problem, we let $\bIQeoV$ be the open subscheme of $\IQe$ parameterizing isotropic subbundles (that is, {saturated} subsheaves), and set
\[
\bIQeV\ := \ \overline{\bIQeoV} ,
\]
the closure of $\bIQeoV$ in $\IQe$. Assume that $\bIQeoV$ is nonempty.
Then the above number $N(g, 2n, \ell, e_0)$ can be realized as the intersection number
\[\int_{\mathbb{IQ}_{e_0}(V)} 1.\]

To compute this, we need to develop an intersection theory on $\bIQeV$. For technical reasons, we will work with labeled schemes $\bIQeV_\delta$, where $\delta \in \{0,1 \}$, such that  $\bIQeV  = \bIQeV_0 \: \cup \: \bIQeV_1$ (see \S\:\ref{subsec:labeling}).
 We first construct some special cycles $\Theta$ on $\bIQeV_\delta$, and then  associate to $\Theta$  the intersection number \[\int_{\bIQeV_\delta}\Theta.\]  
  To get a closed formula for this intersection number, we follow the approach of Holla \cite{Ho}.
  There are two basic operations on orthogonal bundles: Hecke transformation and tensoring by  a line  bundle. We use these operations,  together with  deformations of bundles and curves, to connect $V$ to the trivial orthogonal bundle. For each vector bundle operation, we record a relation between the intersection numbers for $\Theta$  and the `transformed' cycle.
 Note that an intersection number for  the trivial orthogonal bundle turns out to be a Gromov--Witten invariant (\cite{KT1}, \cite{RuTi}). Keeping track of the relations on the intersection numbers  for all the involved operations  taken to connect the original bundle $V$ to the trivial bundle, we can finally express the intersection number  in terms of a Gromov--Witten invariant.


We mention that in some sense our result may  have a relation with a topological quantum field theory (TQFT) of Witten \cite{Wi}. Marian and Oprea studied a TQFT using an intersection theory on Quot schemes, and showed that the so-called Verlinde number of TQFT is equal to an intersection number on a Quot scheme. See also  \cite{Go} for a weighted TQFT.  We believe that if an ``orthogonal Verlinde number" were suitably defined, one could obtain an analogous result. We leave this as a future project.

This paper is organized as follows. In \S\:2, we record some results from \cite{CCH2}  on orthogonal bundles and isotropic Quot schemes. In \S\:3, we review the quantum cohomology of the orthogonal Grassmannian. In \S\:4, we give basic properties of $\bIQeV$ and describe its boundary. In \S\:5, we deal with degeneracy  loci on $\bIQeV_\delta$ and investigate their relations and behavior under Hecke transforms. In \S\:6, we define intersection numbers on  $\bIQeV_\delta$ and obtain relations among intersection numbers on related isotropic Quot schemes. In \S\:7, we give a closed formula for counting maximal isotropic subbundles.

\subsection*{Conventions and notation}
We work throughout over the field $\mathbb C$ of complex numbers. Unless otherwise stated, the curve $C$ may have arbitrary genus $g \ge 0$.
Given a subsheaf $E$ of a locally free sheaf $V$,  we write $\bE$ for the saturation of $E$, which is a subbundle of $V$.

\bigskip

\subsection*{Acknowledgements}
The second named author was supported by the National Research Foundation of Korea (NRF-2020R1F1A1A01068699). The third author thanks the Korea Institute for Advanced Study, the IBS Center for Complex Geometry, Daejeon, and Chungbuk National University for generous financial support and hospitality.

\section{Orthogonal bundles and isotropic subbundles} \label{existence}

In this section, we review some results on orthogonal bundles over $C$ established in \cite{CCH2}.

\begin{definition} Let $L$ be a line bundle. A vector bundle $V$ over $C$ is said to be \textsl{$L$-valued orthogonal} if there is a  nondegenerate symmetric bilinear form $\sigma \colon V \otimes V \to L$; equivalently, if there is a symmetric isomorphism $V \isom V^* \otimes L$. A subbundle, or more generally a subsheaf $E \subset V$ is \textsl{isotropic} if $\sigma|_{E \otimes E} = 0$. \end{definition}

By linear algebra as for orthogonal vector spaces, an isotropic subbundle has rank at most $\frac{1}{2} \rank V$.

In view of the isomorphism $V \isom V^* \otimes L$, we have $\det(V)^2 = L^{\rank V}$. In particular, if $V$ has odd rank, then $\deg L$ is even. If $V$ has even rank $2n$, then $\det (V) = \eta \otimes L^n$ for some $\eta$ of order two in $\Pic^0 (C)$.
The bundle $\eta$ determines whether $V$ admits isotropic subbundles of half rank as follows. The next statement follows from \cite[Lemma 2.5]{CCH2} (see also \cite[Proposition 3.2]{CH4} for a more elementary proof).

\begin{lemma} \label{existence of isotropic bundle} Let $V$ be an $L$-valued orthogonal bundle of rank $2n$. Then $V$ admits an isotropic subbundle of rank $n$ if and only if $\det(V) = L^n$. \end{lemma}

\noindent Concerning the odd rank case, by \cite[Lemma 2.7]{CCH2} we have:

\begin{lemma} \label{existence of isotropic bundle odd rank} Let $V$ be an $L$-valued orthogonal bundle of rank $2n + 1$. Then $V$ admits an isotropic subbundle of rank $n$. \end{lemma}


\subsection{Stiefel--Whitney classes} \label{StWh}

Let $V$ be an $L$-valued orthogonal bundle of rank $r$. When $L = \Oc$ and $\det V$ is trivial, $V$ determines a principal $\SO_r$-bundle (up to a choice of orientation if $r = 2n$ is even; see for example \cite[{\S} 2]{Ra}). In the nonabelian cohomology sequence associated to $1 \to \Z_2 \to \Spin_r \to \SO_r \to 1$, the Stiefel--Whitney class $w_2(V) \in H^2 ( C , \Z_2 )$ is the obstruction to lifting a principal $\SO_r$-bundle to a $\Spin_r$-bundle; see for example \cite{Ser}. As in \cite[{\S} 2.3]{CCH2}, this can be generalized to the case where $\deg(L)$ is even and $\det V = L^n$. In this situation, by for example \cite[Lemma 2.3]{CCH2}, there exists a line bundle $M$ such that $V \otimes M^{-1}$ is $\Oc$-valued orthogonal of trivial determinant.

\begin{definition} Assume that $\deg (L)$ is even and $\det (V) \cong L^n$. We define
\[ w_2(V) \ := \ w_2 (V \otimes M^{-1}) + \frac{n \ell}{2} \ \text{in} \ H^2 (C, \Z_2 ) \cong \Z_2. \]
\end{definition}

\noindent By \cite[Theorem 2.13 (a)]{CCH2}, the invariant $w_2(V)$ is the parity of the degree of any rank $n$ isotropic subbundle of $V$. (See also \cite[Theorem 2.2]{CH4} for a more elementary proof in the case $L = \Oc$.)

\subsection{Zero-dimensional isotropic Quot schemes}
 In \S \:4, we shall discuss the isotropic Quot schemes  parameterizing  isotropic subsheaves in detail. Here we consider a  zero-dimensional isotropic Quot scheme, which will be of interest in this paper.

Let $V$ be an $L$-valued orthogonal bundle of rank $r$ over $C$ of genus $g \ge 2$, and let $\IQ_e(V)$ be the closed subscheme of the Quot scheme of $V$ which parameterizes  isotropic subsheaves of $V$ of rank $n:=\left\lfloor \frac{r}{2} \right\rfloor$ and degree $e$.

\begin{theorem} \label{e_0maximal}
Suppose that an orthogonal bundle $V$  is general in a component of the moduli space $\cM O_C(r ; L)$  of stable $L$-valued orthogonal bundles of rank $r \ge 3$.  We impose the following numerical conditions, with $\ell := \deg (L)$:
\begin{enumerate}
\item If $r=2n$, assume that $\det V \cong L^n$ and that $n(g-1-\ell)$ is even. Set
\[ e_0 \ := \ -\frac{1}{2}n(g-1-\ell) , \]
and assume that $w_2(V) \equiv e_0 \mod 2$.
\item If $r=2n+1$, assume that $\ell$ and $(n+1)(g-1)$ are even. Set
\[ e_0 \ := \ -\frac{1}{2}(n+1)(g-1) + \frac{1}{2}n\ell \]
and assume that $w_2(V) \equiv e_0 \mod 2$.
\end{enumerate}
Then the scheme $\IQ_{e}(V)$ is empty for $e > e_0$, and $\IQ_{e_0}(V)$ is a nonempty smooth scheme of dimension zero, every point of which corresponds to an isotropic subbundle.
\end{theorem}

\begin{proof}
(1) $r=2n$: In the case $L = \Oc$, the (non)emptyness for $e \ge e_0$ follows from \cite[Theorem 1.3]{CH2}. In fact, the same argument works for arbitrary $L$.  For the reader's convenience, let us briefly indicate the steps of the proof.

Let $U_C(n,e)$ be the moduli space of stable bundles over $C$ of rank $n$ and degree $e$. Up to an \'{e}tale cover of $U_C(n,e)$, every orthogonal bundle with a stable degree $e$ isotropic subbundle $E \subset V$ corresponds to a point of a projective bundle $\pi \colon \cP \to U_C(n,e)$ with fibers $\pi^{-1}(E) = \pp H^1(C, \wedge^2 E \otimes L^{-1})$. By dimension count, we have
\[
\begin{cases}
\dim \cP < \dim \cM O_C(2n ; L) & \text{if} \ e > e_0, \\
\dim \cP = \dim \cM O_C(2n ; L) & \text{if}  \ e = e_0.
\end{cases}
\]
 Hence if $e >e_0$, the classifying map $\Phi \colon \cP \dashrightarrow \cM O_C(2n ; L)$ is not dominant and thus $\IQe$ is empty for a general $V$.
 Suppose that $e = e_0$. Now the tangent space to $\IQ_{e_0}(V)$ at $[E \subset V]$ is identified with $H^0(C, L \otimes \wedge^2 E^*) $. We claim that $H^1 ( C, L \otimes \wedge^2 E^* )$ is zero for generic $E \in U_C (n, e_0)$. This follows from \cite[Appendix A]{CH1} for $L = \cO_C$. In general, it suffices to exhibit a single $E$ of rank $n$ and degree $e_0$ with $H^1(C, L \otimes \wedge^2 E^*) =0$. By hypothesis, either $n$ or $g-1-\ell$ is even. When $n$ is even, take $\displaystyle E = \bigoplus_{i=1}^{n/2} F_i$ for general bundles $F_i \in U_C(2, \ell+1-g)$. When $g-1-\ell$ is even, take $E = \bigoplus_{i=1}^{n} M_i$ for general line bundles of degree $\frac{1}{2} (\ell+1-g)$.  Then the vanishing of $H^1(C, L \otimes \wedge^2 E^*)$ can be checked easily.

It follows that
\[
U_C(n, e_0)^\circ \ := \ \{ E \in U_C (n, e_0) \: | \: h^1 ( C, L \otimes \wedge^2 E^* ) = 0 \} .
\]
is a nonempty open subset of $U_C ( n, e_0 )$. For any $E \in U_C (n, e_0)^\circ$, the tangent space $T_E \IQ_{e_0}(V)$ has dimension zero for all $V$ admitting $E$ as an isotropic subbundle. Therefore the restriction of $\Phi$ to $\cP|_{U_C (n, e_0)^\circ}$ is quasi-finite. Hence, by dimension count it is dominant. On the other hand, for dimension reasons the restriction of $\Phi$ to $\cP|_{U_C(n, e_0) \backslash U_C(n, e_0)^\circ}$ cannot be dominant. Thus a general $V \in \MO_C (n; L)$ admits no isotropic subbundle $E$ with $h^1 ( C, \wedge^2 E^* \otimes L ) \ne 0$. It follows that for general $V \in \cM O_C(2n ; L)$ the scheme $\IQ_{e_0} (V)$ is nonempty, smooth and of dimension zero.

%
%

Finally, if $V$ admits a non-saturated subsheaf $[E\subset V] \in \IQ_{e_0}(V)$, then its saturation $\overline{E}$ is an isotropic subbundle of degree $e>e_0$, which is a contradiction if $V$ is general.\\
\par
 (2)  $r=2n+1$: Note that $(\det V)^2 \cong L^{2n+1}$,   thus $\ell$ must be even. The desired result was stated  in \cite[Lemma 8.2]{CCH2}, and its proof is similar to the even rank case.
\end{proof}

\begin{remark}
 There is a typo in the statement of \cite[Theorem 1.5]{CCH2}: the dimension of $\overline{\IQeo}$ is   $I(n+1, \ell, e+\frac{\ell}{2})$, not $I(n+1,  e-\frac{\ell}{2})$ as stated there. The same error appears in the proof of Theorem 1.5 in \S\:8.1. Accordingly, $e-m$ in Lemma 8.2, Proposition 8.3 and its proof must be corrected to $e+m$. Fortunately, this change of sign does not affect the remaining statements and arguments.
\end{remark}

\noindent
\textbf{From now until the end of {\S} \ref{ResultsEven}, we shall assume that our bundles have even rank $2n \ge 4$. The odd rank case will be discussed in {\S} \ref{ResultsOdd}.}

\subsection{Orthogonal Grassmann bundles}

In this subsection we shall work in the following slightly more general situation. Let $\scrC$ be a variety and $\scrL \to \scrC$ a line bundle, and let $\scrV \to \scrC$ be an $\scrL$-valued orthogonal bundle of rank $2n$. Generalizing the orthogonal Grassmannian $\OG(V)$, let $\OG ( \scrV )$ be the closed subvariety of the Grassmann bundle $\Gr ( n, \scrV )$ whose fiber at $x \in \scrC$ is the orthogonal Grassmannian $\OG(n, \scrV_x) \subset \Gr(n, \scrV_x)$.

Before proceeding, we need a definition and a lemma. Let $M$ be an invertible symmetric $r \times r$ matrix defining a nondegenerate symmetric bilinear form on $\bC^r$. We recall from \cite{BG} that the \textsl{conformal orthogonal group} $\GO_r$ is the set
\[ \{ A \in \GL_r \: | \: {^tA} M A = p ( A ) M \hbox{ for some } p ( A ) \in \bC^* \} , \]
It is easy to see that $\GO_r$ is the image of the multiplication map $\bC^* \times \Orth_r \to \GL_r$.
 We write $\GO_r^0$ for the image of $\bC^* \times \SO_r$. Moreover, for $k \ge 1$ we denote the group of $k$th roots of unity by $\bfmu_k$.

\begin{lemma} Fix $n \ge 1$ and suppose $A \in \GO_{2n}$. Then $A \in \GO_{2n}^0$ if and only if $\det ( A ) = p ( A )^n$. \label{technicalGO} \end{lemma}

\begin{proof} The ``only if'' direction is an easy computation.
 Conversely, suppose $A \in \GO_{2n}$ and $\det A = p ( A )^n$. 
 Fix $s \in \bC$ such that $s^{2n} = \det A$. Then for any $\zeta \in \bfmu_{2n}$, the matrix $\zeta^{-1} s^{-1} A$ has determinant $1$. 
 As $A \in \GO_{2n}$, for each $\zeta \in \bfmu_{2n}$ we have
\[ {^t( \zeta^{-1} s^{-1} A )} M ( \zeta^{-1} s^{-1} A ) \ = \ \zeta^{-2} s^{-2} p( A ) M . \]
 Now by hypothesis, $p ( A )^n = s^{2n}$. Hence $p(A) = \xi s^2$ for some $\xi \in \bfmu_n$, and we can find $\zeta \in \bfmu_{2n}$ such that $\zeta^2 = \xi$. Then $A = ( \zeta s ) \cdot ( \zeta^{-1} s^{-1} A )$ belongs to the image of $\bC^* \times \SO_{2n}$, as desired. \end{proof}

Now we can prove the required existence of isotropic subbundles of families, which will be used in the argument for deformation invariance of intersection numbers.

 \begin{proposition} Let $\pi \colon \scrC \to B$ be a family of smooth curves of genus $g$ parameterized by an irreducible curve $B$. Let $\scrV \to \scrC$ be a vector bundle of rank $2n$ and $\Sigma \colon \scrV \isom \scrV^* \otimes \scrL$ a family of nondegenerate quadratic forms. Suppose that $\det \scrV_b \cong \scrL_b^n$ for all $b \in B$. Then for any $b_0 \in B$, there is a neighborhood $U$ of $b_0$ in $B$ and a diagram
\[ \xymatrix{ \tscrC \ar[r]^{\overline{f}} \ar[d]_\tpi & \scrC \ar[d]^\pi \\
 \tU \ar[r]^f & U } \]
where $f$ is an \'etale double cover, such that the orthogonal Grassmann bundle $\OG ( \overline{f}^* \scrV )$ has two connected components. \label{OGTwoCpts} \end{proposition}

\begin{proof} As $\scrV$ is $\scrL$-valued orthogonal, we have $\det \scrV \cong \scrL^n \otimes \Xi$ for some line bundle $\Xi$ of order two in $\Pic ( \scrC )$. Since $\det \scrV_b \cong \scrL_b^n$ for all $b \in B$, the line bundle $\Xi$ is trivial on all fibers $\scrC_b$. Hence $\Xi \cong \pi^* \xi$ for some $\xi$ of order two in $\Pic (B)$. Let $f \colon \tB \to B$ be the \'etale double cover of $B$ associated to $\xi$, and consider the fiber product diagram
\[ \xymatrix{ \tscrC \ar[r]^{\overline{f}} \ar[d]_\tpi & \scrC \ar[d]^\pi \\
 \tB \ar[r]^f & B . } \]

Now $\overline{f}^* \scrV$ is $\overline{f}^* \scrL$-valued orthogonal of determinant $\overline{f}^* \left( \scrL^n \otimes \pi^* \xi \right)$. As $\overline{f}^* \pi^* \xi \cong \tpi^* f^* \xi = \tpi^* \cO_B = \cO_{\tscrC}$, in fact
\begin{equation} \det \overline{f}^* \scrV \ \cong \ \overline{f}^* \scrL^n . \label{detCond} \end{equation}

Now $\overline{f}^* \scrV$ is the associated bundle of a principal $\GO_{2n}$-bundle $\scrE$ over $\tscrC$. By (\ref{detCond}) and Lemma \ref{technicalGO}, the structure group of $\scrE$ further reduces to $\GO_{2n}^0$. Let $P_1$ and $P_2$ be the maximal parabolic subgroups of $\GO_{2n}^0$ preserving $n$-dimensional isotropic subspaces chosen from opposite components of $\OG(n)$. Then
\[ \OG ( \overline{f}^* \scrV ) \ \cong \ \scrE / P_1 \sqcup \scrE / P_2 . \]
As the images of the two loci on the right are disjoint in $\Gr ( n , \scrV )$, we see that $\OG ( \overline{f}^* \scrV )$ has two connected components. \end{proof}


Letting $B$ be the trivial family $\pi_C^* V \to B \times C$ for an orthogonal bundle $V$, we obtain easily:

\begin{corollary}
Let $V \to C$ be an $L$-valued orthogonal bundle of rank $2n$ and determinant $L^n$. Then $\OG (V)$ has two connected components.
\end{corollary}

\section{Quantum cohomology of orthogonal Grassmannians}

In this section, we record some known facts on quantum cohomology of orthogonal Grassmannians. We begin by recalling some basic objects and fixing notation.

\subsection{Notation}\label{subsec:notation}

A \textsl{partition} $\lambda$ is a weakly decreasing sequence of nonnegative integers $\lambda = (\lambda_1 , \lambda_2 , \ldots , \lambda_r)$. The nonzero $\lambda_i$ are called the \textsl{parts} of $\lambda$. The number of parts is called the \textsl{length} of $\lambda$ and is denoted $l(\lambda)$. The sum $\sum_{i=1}^m \lambda_i$ is called the \textsl{weight} of $\lambda$ and is denoted $|\lambda|$.

For a positive integer $m$, let $\cR(m)$ be the set of all partitions $\lambda$ such that $\lambda_1 \leq m$. Such a $\lambda$ is called \textsl{strict} if $\lambda_{i} > \lambda_{i+1}$ for $0 \le i < l(\lambda)$. Denote by $\cD(m)$ the set of all strict partitions in $\cR (m)$.

We shall use the following notation to state the Vafa--Intriligator-type formula in \S\:\ref{GW_for_OG(n)}. For $m = 2r + 1$, set
\[
\cT_m \ := \ \left\{ J = (j_1, \ldots , j_m) \in \Z^m \: | \: -r \leq j_1 < \dots < j_m \leq 3r+1 \right\} ,
\]
and for $m = 2r$, set
\[
\cT_m \ := \ \left\{ J = (j_1 , \ldots , j_m) \in \left( \Z + \frac{1}{2} \right)^m \: {\big|} \: -r +\frac{1}{2} \leq j_1 < \dots < j_m \leq 3r- \frac{1}{2} \right\} .
\]
For $J = (j_1, \ldots ,j_m) \in \cT_m$ and $\zeta := \exp \left( \frac{\pi \sqrt{-1}}{m} \right)$, we write $\zeta^J := \left( \zeta^{j_1}, \ldots ,\zeta^{j_m} \right)$. Define a subset $\cI_m$ of $\cT_m$ by
\[ \cI_m \ := \ \left\{ J = ( j_1 , \ldots , j_m) \in \cT_m\hspace{0.05in}| \: \zeta^{j_k} \ne -\zeta^{j_l}  \text{ for } k \neq l \right\} . \]

\subsection{Symmetric polynomials}\label{subsec:symmetric-poly}

Let $X = ( x_1 , \ldots , x_m )$ be an $m$-tuple of variables. For $1 \le i \le m$, let $h_i (X)$ (resp., $e_i(X)$) be the $i$-th complete (resp., elementary) symmetric function in $x_1 , \ldots , x_m$. By convention, $h_0 (X) = e_0 (X) = 1$ and $h_{k} (X) = e_k (X) = 0$ for $k < 0 $.

The $\Pt$-polynomials of Pragacz and Ratajski \cite{PrRa} are indexed by the elements of $\cR(m)$. We put $\Pt_0(X) =1$ and for  $\lambda = ( a )$ with $1 \le a \le m$, we define
$$ \Pt_{(a)} (X) \ := \ \frac{1}{2}e_a(X) . $$
 Next for $\lambda = ( a , b )$ with $b > 0$, we define
\[
\Pt_{(a,b)}(X) \ :=\ \frac{1}{4} \left( e_a(X) e_b (X) + 2 \sum_{k=1}^b (-1)^k e_{a+k} (X) e_{b-k} (X) \right) .
\]
By convention, $\Pt_{a, 0}(X) =  \Pt_a(X) = \frac{1}{2} e_i(X)$.

For any partition $\lambda \in \cR (m) $ with  $l(\lambda) \geq 3$, let $M_\lambda (X)$ be the $r \times r$ skew-symmetric matrix for $r := 2 \lfloor (l(\lambda)+1)/2 \rfloor$ whose $(i , j)$-th entry  is $\Pt_{(\lambda_i , \lambda_j)}(X)$ for $1 \le i<j  \le r$.

 Then the $\Pt$-polynomial associated to $\lambda$ is defined by
\[ \displaystyle \Pt_{\lambda}(X) \ = \ \textrm{Pfaff} ( M_{\lambda}(X) ) . \]
By convention, for $1 \le k \le m$ we write $\Pt_{k} (X)$ for $\Pt_{(k)} (X)$.

For any partition $\lambda \in \cR(m)$, the Schur polynomial $S_{\lambda}(X)$ is defined by
$$ S_{\lambda}(X) \ := \ \det \left[ h_{\lambda_i+j-i}(X) \right]_{1 \leq i,j \leq m} \ = \ \det \left[ e_{\lambda'_i + j -i}(X) \right]_{1 \leq i,j \leq m} $$
where  $\lambda'$ denotes the conjugate partition of $\lambda$.

Consider the ring $\bC [ x_1 , \ldots , x_m ]^{S_m}$ of symmetric polynomials with coefficients in $\bC$.  Recall that any symmetric polynomial $Q ( x_1 , \ldots , x_m )$ may be expressed as a polynomial in $ \alpha_i: = \frac{1}{2} e_i (X)$ for $1 \le i \le m$. Therefore, in what follows we shall
simply use the notation $Q ( \alpha_1 , \ldots , \alpha_m  )$ for the rewriting of $Q ( x_1 , \ldots , x_m )$  as a  polynomial in $\alpha_1 , \ldots , \alpha_m  $ if no confusion should arise.



\subsection{Cohomology of orthogonal Grassmannians} \label{section:CohomOrthGrass}

The main reference for this subsection is \cite{KT1}.

Let $\Gamma = \bC^{2n}$ be a vector space equipped with a nondegenerate quadratic form $\sigma \colon \Gamma \otimes \Gamma \to \bC$. By linear algebra, maximal isotropic subspaces in $\Gamma$ have dimension $n$. Let $\OG(n, \Gamma)$ or simply $\OG(n)$ denote the orthogonal Grassmannian parameterizing isotropic subspaces of dimension $n$ in $\Gamma$. This is a variety of dimension $\frac{1}{2}n(n-1)$.

\begin{lemma} \label{GH} The orthogonal Grassmannian $\OG(n)$ has two irreducible connected components. Two $n$-dimensional isotropic subspaces $\Lambda_1$ and $\Lambda_2$ belong to the same component if and only if $\dim ( \Lambda_1 \cap \Lambda_2 ) \equiv n \mod 2$. \end{lemma}

\begin{proof} See \cite[Proposition 2, p.\ 735]{GH}. \end{proof}

It is well-known that two components of $\OG(n)$ are isomorphic to each other.  Let us fix one component $\OG(n)_0$. Let $\bfE$ be the tautological subbundle of the trivial bundle $\OG(n)_0 \times \Gamma$ on $\OG(n)_0$.
We  choose a complete isotropic flag in $\Gamma$:
$$ H_{\bullet} : \ H_1 \ \subset \ H_2 \ \subset \ \cdots \ \subset H_{n-1} \subset {H}_n $$
such that $H_n$ belongs to $\OG(n)_0$.
Note that there is another  maximal isotropic subspace $H_n'$  containing $H_{n-1}$ which does not belong to $\OG(n)_0$.

We now write $Z := \OG(n)_0$. For each strict partition $\lambda \in \cD(n-1)$, we define a Schubert variety $Z_{\lambda}(H_\bullet) \subset Z$ by
$$ Z_{\lambda} ( H_\bullet ) \ = \ \left\{ \Sigma \in Z \ \bigg|
\begin{aligned}
\ \rank \left( \Sigma \rightarrow \Gamma / H_{n - \lambda_i}^\perp \right) &\leq  n - i - \lambda_i \hbox{ for } 1 \le i \le l(\lambda)  \\
\ \rank \left( \Sigma \rightarrow \Gamma / \widetilde{H}_{n} \right) &\leq  n - l(\lambda)-1
\end{aligned}
\right\}
$$
where $ \widetilde{H}_{n} = H_n$ if $l(\lambda) \not\equiv n \mod 2$ and $ \widetilde{H}_{n} = H_n'$ if $l(\lambda) \equiv n \mod 2$.

The locus $Z_{\lambda} ( H_\bullet )$ has codimension $|\lambda |$ in $Z$. Let $\tau_\lambda$ be the  cohomology class  $Z_{\lambda}( H_\bullet )$, which is called the \textsl{Schubert class associated to $\lambda$}. The set $\{ \tau_\lambda  \: | \: \lambda \in \cD(n-1) \}$ forms a basis of $H^*(Z)$.
For the variables $X = ( x_1 , \ldots , x_{n-1} )$, we identify $e_i (X) = c_i ( \bfE^\vee )$ for $1 \le i \le n-1$. Then the cohomology class  corresponding to $\Pt_{\lambda} (X)$ coincides with $\tau_\lambda$.

By convention, we write $Z_k ( H_\bullet )$ and $\tau_k$ for the \textsl{special Schubert variety} $Z_{(k)} ( H_\bullet )$ and    the  associated class $\tau_{(k)} \in H^{2k} ( Z )$, respectively.

\subsection{A Vafa--Intriligator-type formula} \label{GW_for_OG(n)}

In this subsection, we state a Vafa--Intriligator-type formula for $\OG(n)_0$, which computes the Gromov--Witten invariants. We begin by defining these invariants.

The \textsl{degree} of a morphism $f \colon C \rightarrow \OG(n)_0$ is defined as the intersection number
\[ \int_{[ \OG(n)_0 ]} f_* [C] \cdot \tau_1 \ =: \ \deg (f) . \]
For the induced isotropic subbundle $E_f$ of the trivial orthogonal bundle $\Oc^{\oplus 2n}$,  we have $\deg ( E_f ) = -2 \deg(f)$.

We now give an informal definition of the Gromov--Witten invariant. For the precise definition, we refer to \cite{RuTi}.

\begin{definition} \label{GW-for-OG} Let $p_1 , \ldots , p_m$ be distinct points of $C$. Let ${\lambda^1}, \ldots , {\lambda^m} \in \cD(n-1)$ be strict partitions. Fix a nonnegative integer $d$.  The \textsl{Gromov--Witten invariant} $\langle \tau_{\lambda^1}, \ldots ,\tau_{\lambda^m} \rangle_{C,d}$ can be informally defined as follows. 
 If
\begin{equation} \sum_{j=1}^m | \lambda^j | \ = \ \frac{1}{2} n ( n - 1 ) ( 1 - g ) + 2 ( n - 1 ) d , \label{SumLengths} \end{equation}
then $\langle \tau_{\lambda^1} , \ldots , \tau_{\lambda^m} \rangle_{C, d}$ is the number of morphisms $f \colon C \rightarrow \OG(n)_0$ of degree $d$ with the property that for each $i$, we have $f(p_i) \in Z_{\lambda^i}(\gamma_i \cdot H_{\bullet})$ for a general $\gamma_i \in \SO_{2n} (\bC)$. If (\ref{SumLengths}) does not hold, then we define $\langle \tau_{\lambda^1}, \ldots ,\tau_{\lambda^m}\rangle_{C,d}$ to be zero. \end{definition}

Since the Gromov--Witten invariant is independent of the points $p_i$ and the curve $C$, depending only on the genus $g$ (see \cite[p.\ 262]{RuTi}), we write $\langle \sigma_{\lambda^1},\ldots , \sigma_{\lambda^m}\rangle_{g, d}$ for $\langle \sigma_{\lambda^1},\ldots , \sigma_{\lambda^m}\rangle_{C, d}$.


The (small) quantum cohomology ring of $\OG(n)_0$ is defined via the genus zero three-point Gromov--Witten invariants \cite{KT1}. Let $q$ be a formal variable of degree $2(n-1)$. The ring $q H^*(\OG(n)_0,\Z)$ is isomorphic as a $\Z [q]$-module to $H^* ( \OG(n)_0 , \Z ) \otimes_{\Z} \Z [q]$. The multiplication in $q H^* ( \OG(n)_0 , \Z )$ is given by the formula
\[
\tau_\lambda * \tau_\mu \ = \ \sum_{d \ge 0 } \sum_\nu \left\langle \tau_\lambda , \tau_\mu, \tau_{\nu^\prime} \right\rangle_{0,d} \tau_\nu \: q^d ,
\]
where $\nu$ ranges over all strict partitions satisfying $|\nu| = |\lambda| + |\mu| - 2(n-1)d$, and as before $\nu'$ denotes the conjugate partition of $\nu$. Note that the specialization of the (complexified) quantum cohomology ring at $q = 1$ is given by
\[
q H^* ( \OG(n)_0 , \bC )_{q = 1} \ := \ q H^* ( \OG(n)_0 , \bC ) \otimes \bC [q] / ( q - 1 ) .
\]
As a complex vector space, this is isomorphic to $H^* ( \OG(n)_0 , \bC)$.

Now we can give a Vafa--Intriligator-type formula for $\OG(n)_0$ for arbitrary genus $g$.

\begin{proposition}\label{corollary:gro;oge} For $m$ partitions $\lambda^1 , \lambda^2 , \ldots , \lambda^m \in \cD(n-1)$ and a nonnegative integer $d$, the genus $g$ Gromov--Witten invariant $\left\langle \tau_{\lambda^1} , \tau_{\lambda^2} , \ldots , \tau_{\lambda^m} \right\rangle_{g,d}$ for $\OG(n)_0$ is given by
$$
\left\langle \tau_{\lambda^1} , \tau_{\lambda^2} , \ldots , \tau_{\lambda^m} \right\rangle_{g,d} \ = \ 4^d \sum_{J \in
\cI_{n-1}} S_{\rho_{n-1}} \left( \zeta^J \right)^{g-1} \Pt_{\lambda^1} \left( \zeta^J \right) \Pt_{\lambda^2} \left( \zeta^J \right) \cdots \Pt_{\lambda^m} \left( \zeta^J \right) $$
whenever (\ref{SumLengths}) holds, and it is zero otherwise.
 \end{proposition}

\begin{proof}
For $g = 0$, the formula was given in \cite{Che1}. For an arbitrary $g$, we have the formula from \cite[p.\:1263]{CMP}:
\be \label{genus-g-VI-formula}
\left\langle \tau_{\lambda^1} , \tau_{\lambda^2}, \ldots , \tau_{{\lambda}^m} \right\rangle_{g,d} \ = \ \mathrm{tr} \left( [ \mathcal{E}^{g-1}\sigma_{\lambda^1} \cdots \sigma_{\lambda^m}] \right) ,
\ee
where $\mathcal{E}$ is the quantum Euler class (cf.\ \cite{Ab}) of $\OG(n)_0$ in $qH^*(\OG(n)_0,\bC)_{q=1}$, and for each $\tau \in qH^*(\OG(n)_0,\bC)_{q=1}$, we denote by $[\tau]$ the corresponding  quantum multiplication operator on $qH^*(\OG(n)_0,\bC)_{q=1}$. Then the formula follows from \cite[Theorem 6.5]{Che2} where the eigenvalues of $[\tau]$ were
explicitly computed for an arbitrary $\tau \in qH^*(\OG(n)_0,\bC)_{q=1}$.  \end{proof}

\section{Isotropic Quot schemes}

\subsection{Generalities}

Let $V$ be an $L$-valued orthogonal bundle of rank $2n$. Let $\ell:= \deg (L)$. For each integer $e$, we define the \textsl{isotropic Quot scheme} $\IQe$ by
\[ \IQeV \ := \ \left\{ [j \colon E \to V] \in \Quot^{n, n\ell - e} (V) \: | \: E \hbox{ isotropic in } V \right\} , \]
a closed subscheme of $\Quot^{n, n \ell - e}(V)$. We denote by $\IQeoV$ the open subscheme consisting of saturated isotropic subsheaves; that is, isotropic subbundles. As discussed in \cite[{\S} 5.1]{CCH2}, in contrast to the classical and symplectic cases, in general the scheme $\IQeV$ is neither irreducible nor equidimensional, and even for $e \ll 0$ can have arbitrarily many components consisting entirely of unsaturated subsheaves. For these reasons, we replace $\IQeV$ with a subscheme which is more tractable and better suited to the enumeration problem that we wish to solve.

\begin{definition} Let $V$ be an orthogonal bundle of degree $n \ell$. We denote by $\bIQeV$ the closure of $\IQeoV$ in $\IQeV$. And we write $\bIQeoV$ for $\IQeo$. Abusing language, we shall also refer to $\bIQeV$ as an ``isotropic Quot scheme''. \end{definition}

\noindent The following is \cite[Theorem 1.4]{CCH2}.

\begin{theorem} \label{Even} Let $V$ be an $L$-valued orthogonal bundle of even rank $2n \ge 4$ and of determinant $L^n$. Then there is an integer $e(V)$ such that:
\begin{enumerate}
\item[(a)] If $\ell$ is even, then for each $e \le e(V)$ with $e \equiv e (V) \mod 2$, the locus $\bIQeV$ has two nonempty connected irreducible components, both of which are generically smooth of dimension
\[
I ( n, \ell, e ) \ := \ ( 1 - n ) e + \frac{1}{2} n ( n - 1 ) ( \ell + 1 - g ) .
\]
Moreover, $\bIQeoV$ is empty when $e \not \equiv e (V) \mod 2$.
\item[(b)] If $\ell$ is odd, then for each $e \le e (V)$, the locus $\bIQeV$ is nonempty and irreducible, and generically smooth of dimension $I(n, \ell, e)$.
\end{enumerate}
\end{theorem}

\noindent As in the symplectic case \cite[Definition 1.1]{CCH1}, we shall initially restrict our attention to isotropic Quot schemes which are well behaved in the following sense.

\begin{definition} \label{Defn Property P} Let $V$ be an $L$-valued orthogonal bundle of rank $2n$. Then we say that $\bIQeV_\delta$ \textsl{has property $\cP$} if every component of $\bIQeV_\delta$ is generically smooth of the expected dimension $I(n, \ell, e)$. We say that $\bIQeV$ has \textsl{property $\cP$} if both $\bIQeV_0$ and $\bIQeV_1$ have property $\cP$. \end{definition}

\subsection{Labeling of components of \texorpdfstring{$\bIQeV$}{IQe(V)}}\label{subsec:labeling}

Let $L$ be a line bundle of degree $\ell$ and $V$ an $L$-valued orthogonal bundle of determinant $L^n$. It follows from Proposition \ref{OGTwoCpts} that the orthogonal Grassmann bundle $\OG(V)$ has two connected components. If $\ell$ is odd, then by \cite[Theorem 2.13 (b)]{CCH2} there is a canonical labeling of the components
\[ \OG(V) \ = \ \OG(V)_0 \ \cup \ \OG(V)_1 , \]
where a rank $n$ isotropic subsheaf of $V$ defines a section of $\OG(V)_\delta$ if and only if $\deg ( \bE ) \equiv \delta \mod 2$. If $\ell$ is even, then by \cite[Theorem 2.13 (a)]{CCH2}, all rank $n$ subbundles of $V$ have degrees of the same parity. In this case, we do not obtain a canonical labeling of the components of $\OG(V)$, so once and for all we fix a labeling $\OG(V) = \OG(V)_0 \cup \OG(V)_1$. This naturally gives rise to labelings $\OG \left( V_p \right) = \OG(V_p)_0 \cup \OG(V_p)_1$ for each $p \in C$, as well as
\[ \bIQeoV \ = \ \bIQeoV_0 \ \cup \ \bIQeoV_1 \quad \hbox{and} \quad \bIQeV \ = \ \bIQeV_0 \ \cup \ \bIQeV_1 . \]
Note that $\bIQeV_\delta$ may be empty for some $\delta$.

For convenience, we define a labeling function $\delta_V \colon \OG(V) \to \mathbb{Z}_2$ which assigns the labeling to each isotropic subspace of  $V$. It should be mentioned that the function $\delta_V$ has nothing to do with  $\delta_W$ for any $W \ncong V$ when $\ell = \frac{2 \deg (V) }{\rank (V)}$ is even.

\begin{remark} \label{Remk:even intersection} Assume that $\bIQeV_\delta$ is not empty. Fix $[ F \to V ] \in \bIQeoV_\delta$. Then by Lemma \ref{GH}, an element $[E \to V] \in \bIQeV$ belongs to $\bIQeV_\delta$ if and only if $\dim \left( \bE_p \cap F_p \right) \ \equiv \ n \mod 2$ for all $x \in C$.
\end{remark}

\subsection{Nonsaturated loci of isotropic Quot schemes} \label{section:NonSat}

For the analogue of property $\cP$ in the symplectic case used in \cite{CCH1}, it was also required that a general point of the isotropic Quot scheme correspond to a saturated subsheaf. In the present situation, this follows from the definition of $\bIQeV$. However, in general $\bIQeV$ does have a nonsaturated boundary, whose structure is analyzed in \cite[{\S} 5.2]{CCH2}. We now examine this more closely.

\begin{definition} \label{TorsionTypeT}
\quad \begin{enumerate}
\item[(a)] We say that a torsion sheaf $\tau$ on $C$ is \textsl{of type} $\cT$ if there is a filtration
\[ 0 \ = \ \tau_0 \ \subset \ \tau_1 \ \subset \ \cdots \ \subset \ \tau_r \ = \ \tau \]
where $\tau_i / \tau_{i-1} \cong \cO_{x_i} \otimes \bC^2$ for some $x_i \in C$ for $1 \le i \le r$. Clearly a torsion sheaf of type $\cT$ has even degree. Note that the filtration may not be unique.
\item[(b)] We say that an element $[ E \to V ]$ of $\IQe$ is \textsl{of type} $\cT$ if either $E \in \IQeoV$ or the quotient $\bE / E$ is of type $\cT$.
\item[(c)] Let $\cF \to B \times C$ be a family of bundles of rank $n$. We write $\cT^{2r} ( \cF )$ for the sublocus of the relative Quot scheme $\cQuot^{0, 2r} ( \cF ) \to B$ of elementary transformations $0 \to E \to \cF_b \to \tau \to 0$ where $\tau$ is of type $\cT$ and length $2r$. It is also a scheme over $B$.
\end{enumerate}
\end{definition}

\noindent The significance of the type $\cT$ property comes from the following (see \cite[Theorem 1.7]{CCH2}).

\begin{theorem} \label{TypeT} Let $V$ be an orthogonal bundle {of even rank $2n \ge 4$}.
\begin{enumerate}
\item[(a)] Every point $[E \to V]$ in $\bIQ_e(V) = \overline{\bIQ_e(V)^\circ}$ is of type $\cT$.
\item[(b)] Let $e(V)$ be as defined in Theorem \ref{Even}. For $e \le e(V)$, every point $[E \to V]$ in $\IQe$ of type $\cT$ lies in $\bIQ_e(V)$.
\end{enumerate}
\end{theorem}

We now study isotropic subsheaves of type $\cT$ more closely, in analogy with \cite[{\S} 1]{Be} and \cite[{\S} 3.4]{CCH1}. Firstly, we introduce an explicit parameter space for elementary transformations of type $\cT$ for use in estimating dimension bounds later. This generalizes and simplifies slightly a construction mentioned in \cite[Lemma 5.3 (a)]{CCH3}.

\begin{proposition} \label{parameter-space} Let $\cF \to B \times C$ be a family of vector bundles of rank $n \ge 2$, where $B$ is an irreducible variety.
\begin{enumerate}
\item[(a)] For each $r \ge 0$, there is an irreducible variety $\cQ^{2r} ( \cF ) \to B$, projective of relative dimension $r(2n-3)$ over $B$, parameterizing flags of subsheaves
\begin{equation} E_r \ \subset \ E_{r-1} \ \subset \ \ \cdots \ \subset \ E_1 \ \subset \ E_0 \ = \ \cF_b \label{filtration} \end{equation}
where for $1 \le i \le r$ we have $E_{i-1} / E_i \cong \cO_{x_i} \otimes \bC^2$ for some $x_i \in C$.
\item[(b)] For $0 \le s \le r$, there is a morphism $\cQ^{2r} ( \cF ) \to \cQ^{2s} ( \cF )$ taking (\ref{filtration}) to the truncated flag $E_s \subset \cdots \subset E_0$.
\item[(c)] There is a classifying map $\cQ^{2r} ( \cF ) \to \cQ uot^{0, 2r} ( \cF )$ taking (\ref{filtration}) to $[ E_r \to \cF_b ]$, which is surjective and generically finite to $\cT^{2r} ( \cF )$.
\item[(d)] There is a universal bundle $\cE_{2r} \to \cQ^{2r} ( \cF ) \times C$ satisfying
\[ \cE_{2r}|_{\{ [ E_r \subset E_{r-1} \subset \cdots \subset \cF_b ] \} \times C} \ \cong \ E_r \]
for each $[ E_r \subset E_{r-1} \subset \cdots \subset \cF_b ] \in \cQ^{2r} ( \cF )$.
\end{enumerate}
\end{proposition}

\begin{proof} We construct the desired objects inductively. We take $\cQ^0 ( \cF ) = B$ and $\cE_0 = \cF$. For $r \ge 1$, let $\cQ^{2r} ( \cF )$ be the sublocus of $\cQ uot^{0, 2} ( \cE_{2r-2} ) \to B$ parameterizing subsheaves $[ E_r \subset E_{r-1} ]$ with
\[ \left[ E_{r-1} \ \subset \ \ \cdots \ \subset \ E_1 \ \subset \ E_0 \ = \ \cF_b \right] \ \in \ \cQ^{2r-2} ( \cF ) \]
and $E_{r-1} / E_r \cong \cO_x \otimes \bC^2$ for some $x \in C$. The locus of such $[ E_r \subset E_{r-1} ]$ is naturally identified with the relative Grassmann bundle $\Gr ( n-2 , \cE_{2r-2} )$, by sending $E_r$ to $\Image \left( E_r|_x \to E_{r-1}|_x \right)$.  As the Grassmannian bundle is irreducible and projective of relative dimension $\dim \Gr ( n-2, n ) + \dim C = 2n-3$ over $\cQ^{2r-1} ( \cF )$, we obtain (a). Part (b) is also clear from the construction.

For (c): The existence of the map $\cQ^{2r} \to \cQ uot^{0, 2r} ( \cF )$  follows from the universal property of Quot schemes. By construction and by Definition \ref{TorsionTypeT}, the image of this map is exactly $\cT^{2r} ( \cF )$. Moreover, if $[E \to \cF_b] \in \cT^{2r} ( \cF )$ is such that $\cF_b / E \cong \cO_D \otimes \bC^2$ for a reduced divisor $D$ of degree $r$, then there are precisely $r!$ filtrations of the form (\ref{filtration}) lying over $[ E \subset \cF_b ]$. Thus the classifying map is generically finite. Lastly, taking $\cE_r$ to be the restriction of the Poincar\'e family on $\cQ uot^{0, 2} ( \cE_{r-1} )$, we obtain (d). \end{proof}

Taking $B$ to be a point, we recover the following (see \cite[Lemma 5.3]{CCH2}).

\begin{corollary} \label{Lemma:TypeT} For a fixed bundle $F$ of rank $n \ge 2$, the locus $\cT^{2r} (F)$ is a closed subset of $\Quot^{0, 2r} ( F )$ which is irreducible and of dimension $r ( 2n - 3 )$. \end{corollary}

Let now $V$ be an $L$-valued orthogonal bundle and $e$ an integer. For fixed $r \ge 0$, let $\cF \to \bIQo_{e+2r} ( V ) \times C$ be the universal bundle. The next statement follows from the definitions and Corollary \ref{Lemma:TypeT}.

\begin{lemma} \label{TrFibration} The association $E \mapsto \bE$ defines a surjective morphism
\[ f_{2r} \colon \cT^{2r} ( \cF ) \ \to \ \bIQo_{e+2r} (V) . \]
If $F \subset V$ is an isotropic subbundle of degree $e+2r$, then $f_{2r}^{-1} ( F )$ is canonically identified with $\cT^{2r} (F)$. 
 In particular, $\cT^{2r} ( \cF )$ is topologically a fiber bundle over $\bIQ_{e+2r}^\circ ( V )$ with irreducible fibers of dimension $r ( 2n - 3 )$. \end{lemma}

\noindent Clearly $\cB_0 = \bIQeoV$ and $f_0$ is the identity.

\subsection{Dominance of evaluation maps}

Let $V$ be an orthogonal bundle of rank $2n$. By \cite[Lemma 7.1]{CCH2} and the definition of $e(V)$ given after \cite[Proposition 7.14]{CCH2}, there exists an integer $f_\delta \ge e(V)$ such that $\ev_{f_\delta, p} \colon \bIQ_{f_\delta}^\circ ( V )_\delta \to \OG( V_p )_\delta$ is dominant for a general $p \in C$. We shall require the following stronger statement.

\begin{lemma} \label{Lemma:EvDom} Let $V$ be an orthogonal bundle of rank $2n$. Suppose $e \le e(V)$ and $e \equiv f_\delta \mod 2$. Then the evaluation map
\[ \ev_{e, p, \delta} \colon \bIQ_e^\circ ( V )_\delta \ \to \ \OG ( V_p )_\delta \]
is dominant for generic $p \in C$. \end{lemma}

\begin{proof} Suppose that $e < f_\delta$ and $e \equiv f_\delta \mod 2$. We claim that the image of
\[ \ev_{e, p, \delta} \colon \overline{\bIQ_e^\circ ( V )_\delta} \ \dashrightarrow \ \OG( V_p )_\delta \]
contains $\Image ( \ev_{f_\delta, p, \delta} )$. For: If $\Lambda = \ev_{f_\delta , p, \delta} ( F )$ then $\Lambda = \ev_{e, p, \delta} ( E )$ for any type $\cT$ elementary transformation $E$ of $F$ such that $p \not\in \Supp ( F/E )$. Since $e \le e(V)$, by Theorem \ref{TypeT} (b) such an $E$ belongs to $\overline{\bIQ_e^\circ ( V )_\delta}$.

But by Theorem \ref{TypeT} (a), the space $\overline{\bIQ_e^\circ ( V )_\delta}$ is irreducible and a general element is saturated. Thus the restriction of $\ev_{e, p, \delta}$ to the dense subset $\bIQ_e^\circ ( V )_\delta$ is also dominant. \end{proof}

\section{Degeneracy loci and Hecke transform}

In this section we define degeneracy loci on the schemes $\bIQeV$ and prove various codimension bounds on the nonsaturated parts of these loci. We then prove various facts we shall require on orthogonal Hecke transformations.

\subsection{Degeneracy loci on \texorpdfstring{$\bIQeV$}{IQe(V)}}\label{subsec:degeneracy}

Let $V$ be an $L$-valued orthogonal bundle. To ease notation, we write $\bX := \bIQeV$. Assume that $\bX_\delta = \bIQeV_\delta$ is nonempty. Write $\pi_C \colon \bX_\delta \times C \to C$ for the projection. Over $\bX_\delta \times C$ there is an exact sequence of sheaves
\begin{equation} 0 \ \rightarrow \ \cE \ \rightarrow \ \pi_C^* V \ \rightarrow \ \widetilde{\cE} , \label{univ} \end{equation}
where $\cE$ is the universal subsheaf and $\widetilde{\cE} = \cE^\vee \otimes \pi_C^*L$. For $p \in C$, denote by $\cE(p)$ and $\cV(p)$ and $\widetilde{\cE}(p)$ the restrictions to $\bX_\delta \cong \bX_\delta \times\{ p \}$ of $\cE$ and $\pi_C^* V$ and $\widetilde{\cE}$, respectively.

Identifying the fibers $\widetilde{\cE}(p) \cong \cE(p)^\vee$,  we obtain an exact sequence over $\bX_\delta$:
\[
0 \ \rightarrow \ \cE(p) \ \rightarrow \ \cV(p) \ \rightarrow \ \cE(p)^\vee ,
\]
where $\cV(p) = \bX_\delta \times V_p$ is the trivial orthogonal bundle.

We will define maximal isotropic degeneracy loci on $\bX_\delta$.
As in {\S} \ref{section:CohomOrthGrass} we fix a complete flag
\begin{equation} H_{\bullet} : \ H_1 \ \subset \ H_{2} \ \subset \ \cdots \ \subset \ H_{n-1} \ \subset \ H_n \label{ReferenceFlag} \end{equation}
of isotropic subspaces of $V_p$ such that $H_n \in \OG ( V_p )_\delta$. As before, we note that there is a unique maximal isotropic subspace $H_n' \in \OG(V_p)_{1 - \delta}$ which contains $H_{n-1}$.  For each $i$, denote by $\cH_i = H_i \times \bX_\delta$ the trivial vector bundle on $\bX_\delta$. For $i=n$, we have two bundles $\cH_n$ and $\cH_n'$.

\begin{definition}\label{def: isotropic loci}
For each $\lambda \in \cD(n-1)$, we define a degeneracy locus  by
\[
 \bX_\lambda ( H_\bullet ; p )_\delta \ = \  \left\{
 \begin{aligned}
{[} \psi \colon E \to V {]} &\in \bX_\delta \: | \:  \\
& \rank \left( \cE(p) \rightarrow \cV(p) / \cH^\perp_{n-\lambda_i} \right)_\psi  \leq  n-i-\lambda_i  \ \text{for } \ 1 \le i \le l(\lambda) \\
&  \rank \left( \cE(p) \rightarrow \cV(p)  / \widetilde{\cH}_{n} \right)_\psi \leq  n - l(\lambda)-1
  \end{aligned}
   \right\} ,
   \]
\end{definition}
\noindent where $ \widetilde{\cH}_{n} = \cH_n$ if $l(\lambda) \not\equiv n \mod 2$ and $ \widetilde{\cH}_{n} = \cH_n'$ if $l(\lambda) \equiv n \mod 2$.


\begin{remark} \label{independent_p} Later we will be mostly concerned with two types of degeneracy loci:


\noindent
 (1) For $\lambda = (k)$ with $1 \le k \le n-1$, we have
 \begin{equation} \label{exact-sequence-p}
 \bX_{(k)}( H_{\bullet} ; p )_\delta =
  \left\{ [ \psi \colon E \to V ] \in \bX_\delta \: \bigg| \:
 \begin{aligned}
 \rank \left( \cE(p) \rightarrow \cV(p) / \cH^\perp_{n-k} \right)_\psi  &\leq  n-k-1  \\
  \rank \left( \cE(p) \rightarrow \cV(p)  / \widetilde{\cH}_{n} \right)_\psi &\leq  n - 2
  \end{aligned}
   \right\}.
   \end{equation}
  This will  be denoted by $\bX_k ( H_{\bullet} ; p )_\delta$.\\
  \par
\noindent (2) For $\lambda = \rho_{n-1}= (n-1, n-2, \dots, 1)$, we have
\be \label{longest degeneracy locus} \bX_{\rho_{n-1}} ( H_{\bullet} ; q )_\delta \ = \ \left\{ [ \psi \colon E \to V ] \in \bX_\delta \: | \:  \psi \left( \cE_q \right) \subseteq \cH_n  \right\}. \ee
Since this depends only on $H_n$, we shall denote it by $\bX_{\rho_{n-1}} ( H_n ; q )_\delta$. \end{remark}

Later we need some bounds on the dimension of certain boundary degeneracy loci.

\begin{lemma} \label{lemma:Ttr} Let $V$ be an orthogonal bundle of rank $2n$ and $F \subset V$ an isotropic subsheaf of rank $n$. Let $p$ be a point of $C$ at which $F$ is saturated. Let $H_\bullet$ be a complete flag in $V_p$ as before. Let $k \in \{ 1 , \ldots , n-1 \}$ be such that $F_p \to V_p / H_{n-k}^\perp$ 
 is surjective. Then the locus
 \[ \{ [ E \to F ] \in \cT^{2r} (F) \: | \: \rank \left( E_p \to V_p / H_{n-k}^\perp \right) \le n-k-1  \ \text{and} \
 \rank \left( E_p \to V_p / \widetilde{H}_n \right) \le n - 2 \} \] \label{ENotSurj}
has codimension at least $k$ in $\cT^{2r} ( F )$. \end{lemma}

\begin{proof} For each $m \ge 0$, let $\cQ^{2m} ( F ) =: \cQ^{2m}$ be as defined in {\S} \ref{section:NonSat}. By Proposition \ref{parameter-space} (c), it will suffice to show that the locus of flags of subsheaves of the form (\ref{filtration}) such that $[ E_r \to F ]$ belongs to (\ref{ENotSurj}) has codimension at least $k$ in $\cQ^{2r}$. For $0 \le m \le r - 1$, we consider the locus
\begin{multline} \left\{ [ E_r \subset \cdots \subset E_1 \subset E_0 = F ] \in \cQ^{2r} \: | \: (E_m)_p \to V_p / H_{n-k}^\perp \hbox{ is surjective} \right. \\
\left. \hbox{and } \left( E_{m+1} \right)_p \to V_p / H_{n-k}^\perp \hbox{ is not surjective} \right\} . \label{mNotSurj} \end{multline}

Now $V_p / H_{n-k}^\perp$ is naturally isomorphic to $H_{n-k}^\vee$. Hence for any element of (\ref{mNotSurj}), we have an exact sequence of vector spaces $0 \to K \to (E_m)_p \to H_{n-k}^\vee \to 0$, where $K \subset (E_m)_p$ is a subspace of dimension $k$. By hypothesis, $\Image \left( (E_{m+1})_p \to (E_m)_p \right) =: \Pi$ is a codimension two subspace of $(E_m)_p$ intersecting $K$ in dimension $k-1$ or $k$. In the first case, we have an exact sequence
\[ 0 \ \to \ \Pi_1 \ \to \ \Pi \ \to \ \Pi_2 \ \to \ 0 \]
where $\Pi_1 \in \Gr ( k-1 , K ) = \pp K^\vee$ and $\Pi_2 \in \Gr ( n - k - 1 , H_{n-k}^\vee ) = \pp H_{n-k}$. The codimension of the locus of such $\Pi$ in $\cT^2 ( E_m ) = \Gr ( n-2 , E_m )$ is therefore at least
\[ \dim \Gr ( n-2, E_m ) - \dim ( \pp K^\vee ) - \dim ( \pp H_{n-k} ) \ = \ (2n - 3) - (k - 1) - (n - k - 1) 
 \ = \ n - 1 . \]
In the second case, we have an exact sequence
\[ 0 \ \to \ K \ \to \ \Pi \ \to \ \Pi' \ \to \ 0 \]
with $\Pi' \in \Gr ( n - k - 2 , H_{n-k}^\vee )$. The codimension of the locus of such $\Pi$ in $\cT^2 ( E_m )$ is therefore at least
\[ \dim \Gr ( n-2, E_m ) - \dim \Gr ( n - k - 2 , H_{n-k}^\vee ) \ = \ ( 2n - 3 ) - (n - k - 2) \cdot 2 
 \ = \ 2k + 1 . \]
For an $E_{m+1}$ of either type, any $[ E_r \to E_{m+1} ] \in \cT^{2(r - m - 1)} ( E_{m+1} )$ defines an element of (\ref{ENotSurj}). Moreover, it follows from Definition \ref{TorsionTypeT} that if $[E \to F]$ is any type $\cT$ elementary transformation with $(F/E)_p$ nonzero, then the rank of the map $E_p \to F_p$ is at most $n-2$. Therefore, every $[E_m \to F ]$ belonging to (\ref{mNotSurj}), and hence any elementary transformation of any such $E_m$, satisfies the second condition $\rank \left( E_p \to V_p \to V_p / \widetilde{H}_n \right) \le n-2$.

In summary, the locus (\ref{ENotSurj}) is of codimension at least $\min \{ 2k+1, n-1 \}$. As $k \le n-1$, we are done. \end{proof}

We now prove a similar statement for the partition $\rho_{n-1}$.

\begin{lemma} Suppose $0 \le m \le n$ and $n - m \equiv 0 \mod 2$. For a fixed $\Lambda \in \OG ( n )_\delta$, the set
\[ Z_{\rho_m} ( \Lambda ) \ = \ \{ \Sigma \in \OG ( n )_\delta \: | \: \dim ( \Sigma \cap \Lambda ) \ge m \} \]
is of codimension at least $\frac{1}{2} m ( m - 1 )$ in $\OG ( n )_\delta$. \label{lemma:codim-m} \end{lemma}

\begin{proof} For $\Sigma \in Z_{\rho_m} ( \Lambda )$, since both $\Sigma$ and $\Lambda$ are isotropic, $\Sigma$ fits into a diagram of the form
\[ \xymatrix{ 0 \ar[r] & \Sigma_1 \ar[d] \ar[r] & \Sigma \ar[d] \ar[r] & \Sigma_2 \ar[d] \ar[r] & 0 \\
 0 \ar[r] & \Sigma_1 \ar[r] & \Sigma_1^\perp \ar[r] & {\Sigma_1^\perp}/{\Sigma_1} \ar[r] & 0 } \]
where $\Sigma_1 \in \Gr ( m , \Lambda )$ and $\Sigma_2 \in \OG \left( {\Sigma_1^\perp}/{\Sigma_1} \right)$. Thus $Z_{\rho_m} ( \Lambda )$ has codimension at least
\begin{multline} \dim \OG ( n ) - \dim \Gr ( m , \bC^n ) - \dim \OG ( n - m ) \ = \\
 \frac{1}{2} n ( n - 1 ) - m ( n - m ) - \frac{1}{2} (n - m ) ( n - m - 1 ) \ = \
 \frac{1}{2} m ( m - 1 ) . \label{codimOne} \qedhere \end{multline}
\end{proof}

\begin{lemma} Let $F \to C$ be a vector bundle of rank $n$. Suppose $0 \le m \le n$ and $n - m \equiv 0 \mod 2$, and let $\Pi$ be an $m$-dimensional subspace of $F_q$ for some $q \in C$. Then the set
\begin{equation} \left\{ [ E \to F ] \in \cT^{2r} ( F ) \: | \: \Image \left( E_q \to F_q \right) \subseteq \Pi \right\} \label{PiInFiber} \end{equation}
is empty or of codimension at least $\frac{1}{2} ( n - m ) ( n + m - 1 )$ in $\cT^{2r} ( F )$.

In particular, given a vector bundle $F \to C$ of rank $n$, the set
\[
\{ [E \to F] \in \cT^{2r}(F) \: | \: E_q \to F_q \ \text{is zero} \}
\]
has codimension at least $\frac{1}{2} n(n-1) $ in $\cT^{2r}(F)$.
 \label{lemma:Pi} \end{lemma}

\begin{proof} We shall prove this by induction on $r$. For $r = 1$, the locus (\ref{PiInFiber}) is empty unless $m = n-2$ (that is, $2r \ge n - m$). In the latter case, (\ref{PiInFiber}) is a single point of $\cT^{2r}$. Thus, as desired, it has codimension
\[ \dim \cT^2 ( F ) \ = \ 2n - 3 \ = \ \frac{1}{2} ( n - ( n - 2 ) ) ( n + ( n - 2 ) - 1 ) . \]

Suppose $r \ge 2$. If $2r < n - m$ then clearly (\ref{PiInFiber}) is empty, so we assume $2r \ge n - m$. We make an induction hypothesis:
\begin{quote}
Suppose $r' < r$. Let $\Pi \subseteq F_p$ be a subspace of dimension $m'$ where $n - m' \equiv 0 \mod 2$. If $2r' \ge n - m'$, then the locus
\[ \{ [ E \to F ] \in \cT^{2r'} ( F ) \: | \: \Image ( E_q \to F_q ) \subseteq \Pi \} \]
has codimension at least $\frac{1}{2} ( n - m' ) ( n + m' - 1 )$ in $\cT^{2r'} ( F )$.
\end{quote}

Now let $[ E \to F ]$ be an element of (\ref{PiInFiber}). We may assume that $E_p \to \Pi$ is surjective, as the locus of $[ E \to F ] \in \cT^{2r} ( F )$ where this is not satisfied is a proper (possibly empty) sublocus of (\ref{PiInFiber}). Then for any filtration
\[ E \ = \ E_r \ \subset E_{r-1} \ \subset \ E_2 \ \subset \ E_1 \ \subset \ E_0 \ = \ F \]
realizing $E$ as a type $\cT$ elementary transformation of $F$, the image of $(E_1)_p \to F_p$ is a codimension two subspace containing $\Pi$. Such $[ E_1 \to F ]$ are parameterized by
\[ \left\{ \Xi \in \Gr ( n-2 , F_p ) \: | \: \Pi \subseteq \Xi \right\} \ \cong \ \Gr \left( n - 2 - m , F_p / \Pi \right) , \]
which has dimension 
 $2 ( n - m - 2)$. Notice that the preimage of $\Pi$ by $E_1 \to F$ is a subspace $\tPi \subseteq (E_1)_p$ of dimension $m + 2$.

Furthermore, $[ E \to E_1 ]$ is an element of $\cT^{2(r-1)} ( E_1 )$ with $\Image \left( ( E_r )_p \to ( E_1 )_p \right)$ contained in $\tPi$. As $2r \ge n - m$, also $2(r-1) \ge n - (m + 2)$. Hence, by induction, we may assume that the locus of such $[ E_r \to E_1 ]$ has codimension
\[ \frac{1}{2} ( n - (m + 2 ) ) ( n + ( m + 2 ) - 1 ) \ = \ \frac{1}{2} ( n - m - 2 ) ( n + m + 1 ) \]
in $\cT^{2(r-1)} ( E_1 )$. Therefore, (\ref{PiInFiber}) has dimension at most
\[ 2 ( n - m - 2 ) + (r - 1)(2n - 3) - \frac{1}{2} ( n - m - 2 ) ( n + m + 1 ) , \]
which one computes is exactly $\dim \cT^{2r} ( F ) - \frac{1}{2} ( n - m ) ( n + m - 1 )$.
 The lemma follows. \end{proof}

\subsection{The Hecke transform}

Let $V$ be an $L$-valued orthogonal bundle of rank $2n$. Let $\Lambda \subset V_q$ be an isotropic subspace of dimension $n$. Let $V^\Lambda$ be the Hecke transform of $V$  with respect to $\Lambda\in \OG(V_q)$ which fits into 
an exact sequence
\be \label{HeckeTransform2} 0 \ \rightarrow \ V \ \stackrel{\psi}{\rightarrow} \ V^\Lambda \ \rightarrow \ \cO_q \otimes \bC^n \ \rightarrow \ 0 . \ee
Furthermore, $V^\Lambda$ is an $L(q)$-valued orthogonal bundle. Restricting this sequence to the fiber at $q$, we get a map $\psi_q : V_q \to (V^\Lambda)_q$ whose kernel is $\Lambda$. Then the image   of $\psi_q$, denoted by $\widetilde{\Lambda}$, is a maximal isotropic subspace of  $(V^\Lambda)_q$ which is canonically isomorphic to $V_q/\Lambda \cong \Lambda^\vee$.    Consider an embedding of ordinary Quot schemes
\be \label{morphism_Quot} \Psi \colon \Quot_{n,e} ( V ) \ \hookrightarrow \ \Quot_{n,e} ( V^\Lambda )  \ee
which sends  $[ E \to V ] \in \Quot_e ( V )$ to $[ \psi (E) \to V^\Lambda ]$.  This restricts to an embedding of isotropic Quot scheme
$$\Psi \colon \bIQeV_\delta \hookrightarrow \mathrm{IQ}_e(V^\Lambda).$$
But the image does not necessarily lie inside $\bIQ_e(V^\Lambda) = \overline{ \mathrm{IQ}^\circ_e(V^\Lambda)}$.   It depends on $\delta_V(\Lambda)$, where $\delta_V: \OG(V) \to \mathbb{Z}_2$ is the labeling function.

\begin{lemma} \label{embedding2} The image of $\Psi$ lies inside $\bIQ_e(V^\Lambda)$ only if $\delta_V ( \Lambda ) \equiv n + \delta \mod 2$. In this case, $\Psi$ embeds $\bIQeV_\delta$ into $\bIQ_e(V^\Lambda)_\epsilon$, where $\epsilon \equiv \delta_{V^\Lambda} \left( \Lambda ^\vee \right) \mod 2$. \end{lemma}

\begin{proof} By Theorem \ref{TypeT} (a), an element $\left[ E \isom \psi ( E ) \to V^\Lambda \right]$ is a limit of isotropic subbundles in $\IQ_e ( V^\Lambda )$ only if $\left[ \psi ( E ) \to \overline{\psi ( E )} \right]$ is of type $\cT$. As saturated subsheaves are dense in $\bIQ_e ( V )$ and $\bIQ_e ( V^\Lambda )$ is closed, it suffices to consider a saturated $[ E \to V ] \in \bIQ_e ( V )$. In this case, $E \subseteq \psi ( E ) \subseteq E \otimes \Oc ( q )$ and
\[ \length \left( \overline{\psi ( E )} / \psi ( E ) \right)_q \ = \ n - \dim \Image \left( \psi ( E )_q \ \to \ V^\Lambda \right) \ = \ \dim ( E_q \cap \Lambda ) . \]
It follows that $[ \psi ( E ) \to V^\Lambda ]$ is of type $\cT$ if and only if
\begin{equation} \dim \left( E_q \cap \Lambda \right) \ \equiv \ 0 \mod \ 2 . \label{congruence} \end{equation}
By Lemma \ref{GH}, this is equivalent to $\delta_V(\Lambda) \equiv n + \delta \mod 2$.

Let $\epsilon \in \Z_2$ be such that $\Psi$ embeds $\bIQeV_\delta$ into $\bIQ_e \left( V^\Lambda \right)_\epsilon$. Then in particular $\epsilon = \delta_{V^\Lambda} \left( \overline{\psi ( E )} \right)$. Now
\[ \dim \left( \overline{\psi ( E )} \cap \Lambda^\vee \right) \ = \ \dim \Image \left( E_q \to \overline{\psi ( E )}_q \right) \ = \ n - \dim \left( E_q \cap \Lambda \right) \ \equiv \ n \mod 2 , \]
the last equivalence in view of (\ref{congruence}). Thus $\delta_{V^\Lambda} \left( \overline{\psi ( E )} \right) = \delta_{V^\Lambda} \left( \Lambda^\vee \right)$ by Lemma \ref{GH}. \end{proof}

To be more general, let $q_1 , \ldots , q_t$ be distinct points of $C$. For $1 \le j \le t$, let $\Lambda_j$ be an element of $\OG \left( V_{q_j} \right)_{\epsilon_j}$ for $\epsilon_j \in \{ 0 , 1 \}$. Let $\tV$ be the orthogonal bundle obtained by taking Hecke transforms of $V$ along $\Lambda_1 , \ldots , \Lambda_t$ respectively. Then $\tV$ fits into the sequence
\begin{equation} \label{Hecke-Sequence}
0 \ \to V \ \to \ \tV \ \to \ \bigoplus_{j=1}^t \cO_{q_j} \otimes \bC^n \ \to \ 0 .
\end{equation}

Then $\deg (\tV) = \deg (V) + tn$. Let $\widetilde{\Lambda}_j \in \OG \left( \tV_{q_j} \right)$ be the image of $V_{q_j}$. As in the case $t = 1$, there is an embedding
\begin{equation} \label{Hecke-embedding}
\Psi \colon \bIQeV_\delta \ \hookrightarrow \ \bIQ_e ( \tV )
\end{equation}
if and only if $\delta_V (\Lambda_j) \equiv n + \delta \mod 2$ for each $j$.\\

The following is an analogue of \cite[Lemma 4.7]{CCH1} in the orthogonal case, but its proof requires a different idea. Recall that property $\cP$ was defined in Definition \ref{Defn Property P}.

\begin{proposition} \label{General;Property P1} Suppose that $n \ge 2$. Let $V \to C$ be an orthogonal bundle of rank $2n$ such that $\bIQeV_\delta$ is nonempty. Let $q_1 , \ldots , q_t$ be distinct points of $C$, and for $1 \le j \le t$ let $\Lambda_j$ be a point of $\OG ( V_{q_j} )_{\delta + n}$. Let $\Psi \colon \bIQeV_\delta \to \bIQ_e (\tV)_\epsilon$ be the embedding (\ref{Hecke-embedding}) for a proper choice of $\epsilon$, where $\tV$ is the Hecke transform along $\Lambda_1 , \ldots , \Lambda_t$. Then there is an integer $t(V)$ such that if $t \ge t(V)$ and the choice of the points $q_j$ and subspaces $\Lambda_j$ is general, then $\bIQ_e (\tV)_\epsilon$ has property $\cP$.
\end{proposition}

\noindent Here, abusing notation slightly, we regard $\delta + n$ as an element of $\Z_2$.

\begin{proof}
By Theorem \ref{Even}, for any orthogonal bundle $W$, there exists $e = e(W)$ such that if $e \le e(W)$, then $\bIQ_e ( W )_\epsilon$ has property $\cP$ for $\epsilon \in \{ 0 , 1 \}$ whenever it is nonempty. The number $e(W)$ depends on $W$, but in fact we can find a uniform such bound for all points in a continuous family of orthogonal bundles. Denote by $\cM O_C ( 2n , n \ell )$ the moduli space of stable orthogonal bundles of rank $2n$ and degree $n \ell \in \{ 0 , n \}$. (Note that
\[
\cM O_C ( 2n , n \ell ) \ = \ \bigcup_{L \in \Pic^\ell (C)} \cM O_C ( 2n ; L )
\]
where $\cM O_C ( 2n ; L )$ is as defined in the proof of Theorem \ref{e_0maximal}.) Then we can find an integer $e_1$ such that for every $W \in \cM O_C ( 2n , n \ell )$, if $e \le e_1$ then $\bIQ_e ( W )_\epsilon$ has property $\cP$ whenever it is nonempty.

For $d = \deg (V)$, the bundle $\tV$ is an orthogonal bundle of degree $d + nt$ which is stable for a general choice of the points $q_j$ and subspaces $\Lambda_j$. Since $d$ is divisible by $n$, there is an integer $m$ such that $d + nt - 2mn \in \{ 0 , n \}$ and $\bIQ_e ( \tV ) \cong \bIQ_{e-mn} ( \tV ( -mp ) )$ by the map
\[
[ E \subset \tV ] \ \mapsto \ \left[ E \otimes \cO_C(-mp) \subset \tV  \otimes \cO_C(-mp) \right] .
\]
Since the bundle $\tV \otimes \Oc ( -mp )$ belongs to $\cM O_C ( 2n , n \ell )$ with $\ell \in \{ 0 , 1 \}$, we see that $\bIQ_e ( \tV ) \cong \bIQ_{e-mn} ( \tV ( -mp ) )_\epsilon$ has property $\cP$ if $e-mn \le e_1$. We can take $t$ large enough so that this holds, and then $\bIQ_e(\tV)_\epsilon$ has property $\cP$.
\end{proof}

\begin{proposition} \label{Prop:Image-degeneracy} Assume that $\bIQeV_\delta$ is nonempty. For  $\Lambda \in \OG(V_p)_\delta$, let
\[
\Psi: \bIQeV_\delta \hookrightarrow \bIQ_e(V^\Lambda)_\epsilon
\]
 be the embedding given in Lemma \ref{embedding2}. Write $\bY=\bIQ_e(V^\Lambda)$.  Then the image of $\Psi$  coincides with the degeneracy locus $\bY_{\rho_{n-1}} ( \Lambda^\vee ; p )_\epsilon\subset \bY_\epsilon.$
\end{proposition}
\begin{proof}
This is immediate from the formula (\ref{longest degeneracy locus}) and the definition of the Hecke transform.
\end{proof}

\begin{corollary}\label{coro:represent} Let $V$ be any orthogonal bundle over $C$, and $\tV$ the Hecke transform given in \eqref{Hecke-Sequence} for isotropic subspaces $\Lambda_j$ for $1 \le j \le t$. Assume that $\bIQeV_\delta$ is nonempty and write $ \bY_\epsilon :=\bIQ_e(\tV)_\epsilon$ which admits an embedding  $\Psi: \bIQeV_\delta  \to \bY_\epsilon$.  Then we have 
 \[
  \Psi\left(\bIQeV_\delta \right ) \ = \ \bigcap_{j=1}^t \bY_{\rho_{n-1}}(\Lambda_j^\vee ; q_j)_{\epsilon}.
  \] 
\end{corollary}


\begin{corollary}\label{Coro:property P}
Let $e$ be an integer. There there is an integer $v(e)$ such that for a general orthogonal bundle $V$ of degree $v\geq v(e)$, the isotropic Quot scheme $\bIQeV$ has property $\cP$ when it is nonempty.
\end{corollary}

\begin{proof}
 Using Proposition \ref{General;Property P1}, this can be proven in essentially the same way as in the Lagrangian case \cite[Corollary 4.9]{CCH1}.
\end{proof}

\section{Gromov-Witten numbers} \label{section:GW numbers}

In this section, we define the so-called Gromov--Witten numbers on $\OG(V)$. These numbers are defined as intersection numbers on $\bIQeV$. We make the definition first for $\bIQeV$ having property $\cP$, and then generalize it to arbitrary $\bIQeV$ using the invariance of intersection numbers under deformations.

\subsection{Intersection numbers on \texorpdfstring{$\bIQeV$}{IQe(V)} having property \texorpdfstring{$\cP$}{P}} \label{subsection:def-GW number}

We remark that unlike type $A$ degeneracy loci, orthogonal degeneracy loci may not represent the corresponding Chern classes of vector bundles. See \cite[\S \: 4.5]{KT2002} for a counterexample. In this situation  we cannot obtain the desired cycle as a product of cycles in the usual manner of intersection theory.  Thus, as in the Lagrangian case, we shall directly construct a cycle necessary in our intersection theory.

Recall that the expected dimension of $\bIQeV$ is
\[ I (n, \ell, e) \ = \ ( 1 - n ) e + \frac{1}{2} n ( n - 1 ) ( \ell + 1 - g ) . \]
More generally, for a nonnegative integer $t$ we set
\[ I_t ( n, \ell, e ) \ = \ I ( n, \ell, e ) - \frac{n(n-1) t}{2}  . \]

The following is an orthogonal analogue of \cite[Proposition 5.1]{CCH1}, whose proof proceeds along similar lines. 

\begin{proposition}\label{general intersection} Let $V$ be an orthogonal bundle of rank $2n$ and degree $n \ell$. Assume that 
 $\bIQeV_\delta$ has property $\cP$. 
 Let $t, k_1 , \ldots , k_s$ be integers with $t \geq 0$ and $1 \leq k_i \leq n-1$, and satisfying $\sum_{i=1}^s k_i = I_t ( n , \ell , e )$. Let $p_1 , \ldots , p_s , q_1 , \ldots , q_t$ be distinct 
 points of $C$. For $1 \le i \le s$, let $H_{i, \bullet}$ be a complete isotropic flag in $V_{p_i}$. 
 For $1 \le j \le t$, let $\Lambda_j \subset V_{q_j}$ be an isotropic subspace of dimension $n$ belonging to $\OG ( V_{q_j} )_\delta$. Then for general choices of $\gamma_i \in \SO(V_{p_i})$ and $\eta_j \in \SO(V_{q_j})$, the following hold.
\begin{enumerate}
\item[(a)] The intersection
\begin{equation} \label{intersection1} \left( \bigcap _{i=1}^s \bX_{k_i} ( \gamma_i \cdot H_{i, \bullet} ; p_i )_\delta \right) \cap \left( \bigcap_{j=1}^t \bX_{\rho_{n-1}} ( \eta_j \cdot \Lambda_j ; q_j )_\delta \right) \end{equation}
 is empty or of dimension zero in $\bIQeV_\delta$.
\item[(b)] The intersection \eqref{intersection1} is equal to
\begin{equation}
\left( \bigcap_{i=1}^s \bX_{k_i}^\circ ( \gamma_i \cdot H_{i, \bullet} ; p_i )_\delta \right) \cap \left( \bigcap_{j=1}^t \bX_{\rho_{n-1}}^\circ ( \eta_j\cdot \Lambda_j ; q_j )_\delta \right) \label{intersection2}
\end{equation}
In particular, every element corresponds to a saturated subsheaf.
\item[(c)] The scheme \eqref{intersection1} is reduced.
\end{enumerate}
\end{proposition}

\begin{proof} Firstly, we claim that for a general choice of $\gamma_i$ and $\eta_j$, the locus (\ref{intersection1}) is empty or has dimension zero. To see this, note that the group $\prod_{i=1}^s \SO ( V_{p_i} ) \times \prod_{j=1}^t \SO ( V_{q_j} )$ acts transitively on $\prod_{i=1}^s \OG ( V_{p_i} )_{\delta + n} \times \prod_{j=1}^t \OG ( V_{q_j} )_\delta$. (Note that $\OG(V)_{\delta + n}$ is the component containing subspaces and subbundles intersecting $\Lambda$ in dimension zero.) Now each $Z_{k_i} ( H_{i, \bullet} )$ has codimension $k_i$ in $\OG ( V_{p_i} )_{\delta + n}$, and each $Z_{\rho_{n-1}} ( \Lambda_j )$ has codimension $\frac{1}{2}n(n-1)$ in $\OG ( V_{q_j} )_\delta$.

Hence by \cite[Theorem 2 (i)]{Kle}, the intersection (\ref{intersection1}) is empty or else equidimensional and of codimension
\[
\sum_{i=1}^s k_i + t \cdot \frac{1}{2}n(n-1) \ = \ I ( n, \ell, e)
\]
in $\IQeo_\delta$. As $\bIQe_\delta$ has property $\cP$, we obtain the claim.

In view of the claim, statements (a) and (b) would follow if we show that the intersection points in  (\ref{intersection1}) are contained in the saturated part $\IQeo_\delta$. We shall do this by showing that the nonsaturated part of (\ref{intersection1}) is of negative expected dimension, and then use property $\cP$ together with Kleiman's theorem.

Let $I$ be a subset of $\{ 1, \ldots , s \}$, and let $m_1 , \ldots , m_t$ be integers with $0 \le m_j \le n$ and $n - m_j \equiv 0 \mod 2$ for each $j$. Consider the set
\begin{multline} \left\{ [ F \to V ] \in \IQo_{e+2r} ( V )_\delta \: | \: F_{p_i} \to V_{p_i} / \left( \gamma_i \cdot H_{i , n-k_i}^\perp \right) \hbox{ not surjective for } i \in I \right. \\
 \left. \hbox{ and surjective for } i \not\in I , \hbox{ and } \dim ( F_{q_j} \cap \eta_j \cdot \Lambda_j ) = m_j \right\} \label{partialIntersBase} \end{multline}
Now by Remark \ref{independent_p}, for any such $[F \to V]$ and for $i \in I$, necessarily the map $F_{p_i} \to V_{p_i} / H_{i, n}$ has rank at most $n-2$. It follows that (\ref{partialIntersBase}) is precisely the inverse image in $\IQo_{e + 2r} ( V )_\delta$ of
\[ \left( \prod_{i \in I} Z_{k_i} ( V_{p_i} ; \gamma_i \cdot H_{i, \bullet} ) \right) \times \left( \prod_{j=1}^t Z_{\rho_{m_j}} ( \eta_j \cdot \Lambda_j ; V_{q_j} ) \right) \]
by the product of evaluation maps
\[ \left( \prod_{i \in I} \ev_{e + 2r , p_i , \delta} \right) \times \left( \prod_{j = 1}^t \ev_{e + 2r, q_j , \delta} \right) . \]
By \cite{KT1}, Lemma \ref{lemma:codim-m} and \cite[Theorem 2 (i)]{Kle}, for a general choice of $\gamma_i$ and $\eta_j$ the locus (\ref{partialIntersBase}) is empty or of codimension
\begin{equation} \sum_{i \in I} k_i + \sum_{j=1}^t \frac{1}{2} m_j ( m_j - 1 ) \label{codimBase} \end{equation}
in $\IQoer$.

Next, for any $[F \to V]$ in (\ref{partialIntersBase}), we consider the locus
\begin{multline} \left\{ [ E \to F ] \in \cT^{2r} ( F ) \: | \: \left( E_{p_i} \to F_{p_i} \to V_{p_i} / \left( \gamma_i \cdot H_{i, n-k_i}^\perp \right) \right) \hbox{ is not surjective for $i \not\in I$, and} \right. \\
 \left. \Image \left( E_{q_j} \to F_{q_j} \right) \subseteq F_{q_j} \cap ( \eta_j \cdot \Lambda_j ) \right\} \label{partialIntersFiber} \end{multline}
By the proof of Lemma \ref{lemma:Ttr}, for any such $[E \to F]$, for $i \not\in I$ the rank of the map $E_{p_i} \to V_{p_i} / \gamma_i \cdot H_{i, n}$ is at most $n - 2$. Hence such an $E$ defines a point of $\bX_{k_i} ( \gamma_i \cdot H_{i, \bullet} ; p_i )$ for all $i \not\in I$ also.

By Lemma \ref{lemma:Ttr} and Lemma \ref{lemma:Pi}, and since the $p_i$ and $q_j$ are distinct, for any $\gamma_i$ and $\eta_j$ the locus (\ref{partialIntersFiber}) is empty or of codimension
\begin{equation} \sum_{i \not\in I} k_i + \sum_{j=1}^t \frac{1}{2} ( n - m_j ) ( n + m_j - 1 ) . \label{codimFiber} \end{equation}
(Note that as pointed out in the proof of Lemma \ref{PiInFiber}, a necessary condition for nonemptiness of (\ref{partialIntersFiber}) is that $2r \ge \sum_{j=1}^t ( n - m_j )$.)

Now since the conditions expressed in (\ref{partialIntersBase}) and (\ref{partialIntersFiber}) are purely on the base and the fibers of $\cT^{2r} \to \IQoer$, to compute the codimension of the component of the intersection of $\cT^{2r} (V)$ with (\ref{intersection1}) corresponding to this choice of $I \subseteq \{ 1 , \ldots , s \}$ and $m_1 , \ldots , m_t$, we can simply add the codimensions (\ref{codimBase}) and (\ref{codimFiber}). We obtain
\[ \sum_{i \in I} k_i + \sum_{i \not\in I} k_i + \sum_{j=1}^t \frac{1}{2} \left( m_j ( m_j - 1 ) + ( n - m_j ) ( n + m_j - 1 ) \right) , \]
which is exactly $I ( n, \ell , e )$.

But since  $\bIQeV_\delta$ has property $\cP$, the nonsaturated locus is of dimension strictly less than $I ( n, \ell , e )$. As there are at most finitely many possible values of $r$, and finitely many possible $I$ and $m_1 , \ldots , m_t$ for a given $r$, it follows that for a general choice of $\gamma_i$ and $\eta_j$, the intersection of (\ref{intersection1}) with the nonsaturated locus is empty.

Part (c) can be proven in the same way as \cite[Proposition 5.1 (3)]{CCH1}. \end{proof}

Recall from {\S} \ref{subsec:symmetric-poly} that for $i = 1 , \ldots , n-1$ we have defined $\alpha_i = \frac{1}{2} e_i ( X )$, and made a convention to express symmetric polynomials in $ x_1 , \ldots , x_{n-1} $ in terms of the  polynomials $\alpha_1, \dots, \alpha_{n-1}$. Thus a symmetric polynomial $Q$ will appear below as $Q(\alpha)$ for $\alpha =  (\alpha_1, \dots, \alpha_{n-1})$.


\begin{definition} \label{different expression} Let $m$ be a nonnegative integer and $Q = Q(\alpha)$ a homogeneous polynomial of (weighted) degree $I_m ( n, \ell, e )$. Then we define $\Theta( Q ; \, t)$ as follows. For integers $t, k_1, \ldots , k_s$ as in Proposition \ref{general intersection}, if $Q ( \alpha )$ is a monomial $\prod_{i=1}^s \alpha_{k_i}$, then we define $\Theta ( Q ; \, t )$ as the zero-cycle determined by the intersection \eqref{intersection1}. Then we can extend this definition of $\Theta(Q ; \, t)$ to an arbitrary (weighted) homogeneous polynomial $Q ( \alpha )$, by linearity of cycles. \end{definition}

\noindent Note that in Definition \ref{different expression}, we defined $\Theta(Q ; \, t)$ for fixed $p_i$, $q_j$, $H_{i, \bullet}$ and $\Lambda_j$, and general $\gamma_i$ and $\eta_j$.

\begin{definition} Let $V$ be an $L$-valued orthogonal bundle of degree $v = n \ell$. Suppose that $\bX_\delta = \bIQeV_\delta$ has property $\cP$. For a nonnegative integer $t$ and a homogeneous polynomial  $Q = Q ( \alpha )$ of degree $I_t ( n, e, \ell )$, define
\begin{equation}
N_{C, e}^v ( Q ; \, t ; \, V)_\delta \ := \ \int_{\bX_\delta} \Theta ( Q ; \, t ) .
\end{equation}  By convention, we set $N_{C, e}^v ( Q ; \, t ; \, V)_\delta:=0 $ if  $\deg Q(\alpha) \ne  I_t ( n, e, \ell ).$
When $t = 0$, we write $N_{C, e}^v ( Q ; \, V)$ for $N_{C, e}^v(Q ; \, 0 ; \, V )$. \end{definition}

\noindent We shall refer to the numbers $N_{C, e}^v(Q ; \, t ; \, V)_\delta$ as \textsl{Gromov--Witten numbers of $\OG(V)_\delta$}.

\begin{remark} \label{remark:intersection} By Proposition \ref{general intersection}, we can keep the intersection \eqref{intersection1} zero-dimensional as $p_i$, $q_j$, $H_{i, \bullet}$ and $\Lambda_j$ vary as long as the $\gamma_i$ and $\eta_j$ are general. Thus, by \cite[Proposition 10.2]{Fu}, the number $N_{C,e}^v ( Q ; V )_\delta$ does not depend on the chosen $p_i$, $q_j$, $H_{i, \bullet}$ and $\Lambda_j$. \end{remark}

\subsection{Invariance of intersection numbers under deformations} \label{section:invariance}

We now show that the Gromov--Witten number defined in the previous subsection is invariant under small deformations of the curve and bundle.


Let $\scrC \to B$ be a family of smooth projective curves over an irreducible curve $B$, and $\scrL \to \scrC$ a line bundle of relative degree $\ell$. Let $\scrV \to \scrC$ be an $\scrL$-valued orthogonal bundle of rank $2n$ satisfying $\det \scrV_b \cong \scrL_b^n$ for each $b \in B$. By Proposition \ref{OGTwoCpts}, after passing to an \'etale double cover of $B$, the orthogonal Grassmannian $\OG ( \scrV )$ has two components, which we endow with a labeling
\be
\label{labeling} \OG(\scrV) \ = \ \OG(\scrV)_0 \cup \OG(\scrV)_1 .
\ee
Consider now the relative isotropic Quot scheme $\bIQ_e ( \scrV ) \to B$. Via the evaluation map $\bIQ_e^\circ ( \scrV ) \to \OG ( (\scrV_b )_x )$ 
 for any $b \in B$ and $x \in \scrC_b$, we obtain a decomposition, possibly with an empty component:
$$
\bIQ_e ( \scrV ) \ = \ \bIQ_e ( \scrV )_0 \cup \bIQ_e( \scrV )_1 .
$$
Assume that $\bIQ_e (\mathscr V)_\delta$ is nonempty. 
 Let $b$ be a point of $B$, and $U \subset B$ a neighborhood of $b$. Let $\overline{p} \colon U \to \scrC$ be a local section of the family $\scrC \to B$. Write $\scrV_U$ for the restriction of $\scrV$ to $\scrC_U$. For $1 \le k \le n$, let $\overline{H_k} \colon U \to \Gr ( k , \scrV_U )$ be sections such that for each $b \in U$,
\[ \overline{H_1}|_b  \ \subset \ \cdots \ \subset \ \overline{H_n} |_b  \]
is a flag of isotropic subspaces with $\overline{H_n} |_b \in \OG ( \scrV_{\tilde{p} ( b )} )_\dpr$. Then for each $\lambda \in \cD(n-1)$, as in the Lagrangian case \cite[{\S} 5.2]{CCH1}, one can define a degeneracy locus $\overline{\bX}_\lambda ( \overline{H_\bullet} ; \overline{p})_\delta$ on $\bIQ( \scrV_U)_\delta$. The degeneracy locus $\overline{\bX}_\lambda (\overline{H_\bullet};\overline{p})$ can be viewed as a family of degeneracy loci parametrized by $U$.

Now assume that $\bIQ_e ( \scrV_b )$ has property $\cP$ for each $b \in B$. Let $t \geq 0$ be an integer and $Q ( \alpha ) = \prod_{i=1}^s \alpha_{k_i}$ a monomial with $\deg Q ( \alpha ) = I_t ( n , \ell , e)$. Then, choosing more sections and flags, we can form a relative version of \ref{intersection1}, a family of the intersections \eqref{intersection1} parameterized by $U$. If all the sections and flags are sufficiently general, as $\dim U = \dim B = 1$, this family is a one-dimensional subscheme of $\widetilde{\bX}_\delta|_U$; and moreover, each component dominates $U$. Thus, for each $b_1 \in U$, it determines a $0$-cycle $\Theta(Q ; \, t)$ of $\bIQ_e ( \scrV|_{b_1} )_\delta$. Therefore, by \cite[Proposition 10.2]{Fu}, we now obtain the desired invariance property:

\begin{proposition} \label{Invariance1} Let $\scrC \to B$ and $\scrV \to \scrC$ be as above. For an integer $t \geq 0$ and a homogeneous polynomial $Q ( \alpha )$ of degree $I_t ( n, \ell, e )$, the intersection number $N_{\scrC_b , e}^v(Q ; \, t ; \, \mathscr V_b)_\delta$ is independent of $b \in B$. \end{proposition}

\subsection{Intersection numbers on arbitrary isotropic Quot scheme}

In \S\:\ref{subsection:def-GW number}, we defined intersection numbers on isotropic Quot schemes having property $\cP$. We shall now define intersection numbers on arbitrary isotropic Quot schemes.

Let $V$ be an orthogonal bundle and $\tV$ a Hecke transform of $V$ as in \eqref{Hecke-Sequence} with respect to subspaces $\Lambda_j$  belonging to the component $\OG ( V_{q_j} )_\delta$ for $1 \le j \le t$. So we have the embedding
\[
\Psi \colon \bX_\delta \ := \ \bIQeV_\delta \ \hookrightarrow \ \bIQ_e(\tV)_\epsilon \ =: \ \bY_\epsilon ,
\]
where $\epsilon = \delta_{\tV} ( \Lambda^\vee_j )$.

Let $p$ be a point of $C \setminus \{ q_1 , \ldots , q_t \}$, and $H_\bullet$ a complete isotropic flag in $V_p = \tV_p$. Then for $\lambda \in \cD(n-1)$, it is obvious that the preimage of $\bY_\lambda ( H_\bullet ; p)$ by $\Psi$ is equal to the degeneracy locus $\bX_\lambda ( H_\bullet ; p )$ on $\bX$; or, equivalently, the equality
\begin{equation} \label{equality-subscheme}
\bY_\lambda ( H_\bullet ; p )_\epsilon \cap \Psi \left( \bX_\delta \right) \ = \ \Psi \left( \bX_\lambda ( H_\bullet ; p )_\delta \right)
\end{equation}
holds. Thus, using the equality in Corollary \ref{coro:represent}:
\[
\Psi \left( \bX_\delta \right) \ = \ \bigcap_{j=1}^t \bY_{\rho_{n-1}} ( \Lambda^\vee_j ; q_j )_\epsilon ,
\]
we obtain the following.

\begin{lemma} Let $V$ and $\tV$ be as above. Let $p_1 , \ldots , p_m$ be distinct points in $C$ such that the sets $\{ p_i \}$ and $\{ q_j \}$ are distinct. For $1 \le i \le m$, let $H_{i, \bullet}$ be a complete isotropic flag in $V_{p_i} = \tV_{p_i}$. For strict partitions $\lambda^1 , \ldots , \lambda^m \in \cD(n-1)$, we have an isomorphism of schemes
\begin{equation} \label{isom-subschemes}
\bigcap_{i=1}^m \bX_{\lambda^i} ( H_{i, \bullet} ; p_i)_\delta \ \stackrel{\sim}{\rightarrow} \ \left( \bigcap_{i=1}^m \bY_{\lambda^i} ( H_{i, \bullet} ; \, p_i )_\epsilon \right) \cap \left( \bigcap_{j=1}^t \bY_{\rho_{n-1}} ( \Lambda^\vee_j ;\, q_j )_\epsilon \right)
\end{equation}
\end{lemma}

\begin{corollary}\label{Coro:equality-numbers}
Let $V$ be an orthogonal bundle of degree $v = n \ell$, and $\tV$ the Hecke transform of $V$ at points $q_1 , \ldots , q_t\in C$ as above. Assume that $\bIQeV_\delta$ and $\bIQ_e ( \tV )_\epsilon$ have property $\cP$. Let $u$ be a nonnegative integer and $Q ( \alpha )$ a homogeneous polynomial of degree $ I_u ( n , \ell , e )$. Then we have
\begin{equation} \label{equality-numbers}
N_{C, e}^v \left( Q ( \alpha ) ; u ; V\right)_\delta \ = \ N_{C, e}^{v + tn} \left( Q ( \alpha ) , u + t ; \tV \right)_\epsilon .
\end{equation}
\end{corollary}

\begin{proof}
By linearity of intersection numbers, it suffices to consider the case when $Q ( \alpha )$ is a monomial of the form $\prod_{i=1}^s\alpha_{k_i}$. Set $m = u + s$. Let $p_1 , \ldots , p_m , q_1 , \ldots , q_t$ and $H_{i, \bullet}$ and $\Lambda_j$ be as above for $1 \le i \le m$ and $1 \le j \le t$. Let $\gamma_i \in \SO ( V_{p_i} )$ for $1 \le i \le m$ and $\eta_j \in \SO (V_{q_j} )$ for $1 \le j \le t$.

Now take $\lambda^i = ( k_i )$ for $1 \le i \le s$ and $\lambda^i = \rho_{n-1}$ for $s + 1 \le i \le s + u = m$. Then by Proposition \ref{general intersection}, for general $\gamma_i$ and $\eta_j$, the intersection in the righthand side of \eqref{isom-subschemes} gives a zero-dimensional scheme determining a cycle $\Theta ( Q ( \alpha ) ; u + t )$ on $\bY_\epsilon$. Recall that $N_{C, e}^{v + tn} \left( Q ( \alpha ) ; u + t ; \tV \right)$ is the degree of this cycle. Note that for a general choice of $\gamma_1 , \ldots , \gamma_m$, the cycle with $\eta_j = \Iden \in \SO ( V_{q_j} )$ for $1 \le j \le t$ is obtained as a limit of general zero-cycles $\Theta ( Q ( \alpha ) ; u + t )$ on $\bY_\delta$ as $\eta_1 , \ldots , \eta_t$ vary. Since the degree does not change under deformation of zero-cycles, the degree of this limit cycle is equal to $N_{C, e}^{v + tn} \left( Q ( \alpha ) , u + t ; \tV \right)_\epsilon$. On the other hand, the restriction of this limit cycle to the subscheme $\bX_\delta$ of $\bY_\epsilon$ is equal to the zero-cycle $\Theta ( Q ( \alpha ) ; u )_\delta$ on $\bX_\delta$ whose degree is $N_{C, e}^v \left( Q ( \alpha ) , u ; V \right)_\delta$. Hence we obtain the equality \eqref{equality-numbers}.
\end{proof}




Now we define an intersection number  $\bIQeV_\delta$ in general, without the condition of  property $\cP$.

\begin{definition} \label{Def:Arbitrary}
Let $V$ be an orthogonal bundle of degree $v = n \ell$, and suppose $\bIQeV_\delta$ is nonempty. Let $\tV$ be  the orthogonal Hecke transform of $V$ with respect to a general choice $\Lambda_1 , \ldots , \Lambda_t$ belonging to a component $\OG ( V )_\delta$. Then $\deg ( \tV / V ) = t n$, and by Proposition \ref{General;Property P1} we may assume that $\bIQetV_\delta$ has property $\cP$ for $t \gg 0$. Let $u$ be a nonnegative integer and $Q ( \alpha )$ a homogeneous polynomial of degree $I_u ( n , \ell , e )$. Then we define
$$
\tN_{C, e}^v ( Q ( \alpha ) ; u ; V )_\delta \ := \ N_{C, e}^{v + tn} ( Q ( \alpha ) ; u + t ; \tV )_\epsilon ,
$$
where $\epsilon = \delta_{\tV} ( \Lambda_j^\vee )$.
When $u = 0$, we write $\tN_{C, e}^v ( Q ( \alpha ) ; V )_\delta$ for $\tN_{C, e}^v ( Q ( \alpha ) ; 0 ; V )_\delta$.
\end{definition}

\begin{proposition} \label{prop: Ntilde} The number $\tN_{C, e}^v ( Q ( \alpha ) ; u ; V )_\delta$ is well defined and depends only on $g$, $e$ and $v$ once the polynomial $Q ( \alpha )$ is specified. More precisely,
\begin{enumerate}
\renewcommand{\labelenumi}{(\arabic{enumi})}
\item The number $\tN_{C, e}^v ( Q ( \alpha ) ; u ; V )_\delta$  does not depend on the chosen Hecke transform $\tV$.
\item Let $\scrV \to \scrC \to B$ be a family of curves and orthogonal bundles parameterized by a connected curve $B$. 
 Then $\tN_{\cC_b , e}^v \left( Q ; u ; \scrV|_{\cC_b} \right)_\delta$ is constant with respect to $b \in B$. (In particular, it is invariant even for not necessarily flat families of isotropic Quot schemes.)
\end{enumerate}
\end{proposition}

\begin{proof} For simplicity, for both (1) and (2), we prove the proposition for $u = 0$. The proof for the general case is not different.

(1) Take two different general Hecke transforms $\tV_1$ and $\tV_2$ of $V$ as in Corollary \ref{Coro:property P}.  We may assume that the Hecke transforms are obtained at $t_1+t_2$ subspaces $\Lambda_i \in \OG ( V_{q_i} )_\delta$ for $1 \le i \le t_1$ and $\Omega_j \in \OG ( V_{r_j})_\delta$ for $1 \le j \le t_2$, where the points $q_1, \dots, q_{t_1}, r_1, \dots, r_{t_2}$ are all distinct. We have the corresponding embeddings
$$
\bIQeV_\delta\hookrightarrow \bIQ_e(\tV_1)_{\epsilon_{1}}\ \ \mathrm{and}\ \ \bIQeV_{\delta}\hookrightarrow \bIQ_e(\tV_2)_{\epsilon_2} ,
$$
for suitable $\epsilon_1, \epsilon_2 \in \Z_2$  where moreover $\bIQ_e ( \tV_1 )_{\epsilon_1}$ and $\bIQ_e( \tV_2 )_{\epsilon_2}$ have property $\cP$. Now take a Hecke transform of $\tV_1$ at the isotropic subspaces $\Omega_j$ of $( \tV_1 )_{r_j} = V_{r_j}$ for $1 \le j \le t_2$, and also a Hecke transform of $\tV_2$ at the isotropic subspaces $\Lambda_i$ of $( \tV_2 )_{q_i} = V_{q_i}$ for $1 \le i \le t_1$ to obtain a orthogonal bundle $\tV_3$ which is a common Hecke transform of $\tV_1$ and $\tV_2$. Note that for the given labeling $\delta$, we have $\Lambda_i \in \OG(V_{q_i})_\delta$ and $\Omega_j \in \OG(V_{r_j})_\delta$ for $1 \le i \le t_1$ and $1 \le j \le t_2$. Thus by (\ref{Hecke-Sequence}), for some $\epsilon_3 \in \Z_2$ we have
$$
\bIQ_e ( \tV_1 )_{\epsilon_1} \ \hookrightarrow \ \bIQ_e ( \tV_3 )_{\epsilon_3} \ \ \hbox{and} \ \ \bIQ_e ( \tV_2 )_{\epsilon_2} \ \hookrightarrow \ \bIQ_e(\tV_3)_{\epsilon_3} .
$$
By Corollary \ref{Coro:property P}, we may assume that the isotropic Quot schemes of all the intermediate Hecke transforms in the construction above between $\tV_i$ and $\tV_3$ also have property $\cP$. Then, by Corollary \ref{Coro:equality-numbers},  we obtain
\[
N_{C,e}^{v + t_1 n}( Q(\alpha) ; t_1; \tV_1)_{\epsilon_1} \: = \: N_{C,e}^{v + (t_1+t_2) n}( Q(\alpha) ; t_1+t_2; \tV_3)_{\epsilon_3} \: =\: N_{C,e}^{v + t_2 n}( Q(\alpha);t_2; \tV_2)_{\epsilon_2},
\]
 which proves $(1)$.

(2) Let $b_0$ be a point of $B$. Take $t \geq t \left( \scrV|_{b_0} \right)$, so that $\bIQ_e \left( \overline{\scrV|_{b_0}} \right)_\delta$ has property $\cP.$ For $1 \le j \leq t$, let $\overline{q_j} \colon U \to \cC$ be a local section of $\cC \to B$, and let $\overline{\Lambda_j} \colon U \to \OG \left( \scrV|_{\overline{q_j}} \right)_\delta$ be a local section of the composition $\OG \left( \scrV|_{\overline{q_j}} \right)_\delta \to \cC \to B$ for some open subset $U \subseteq B$ containing $b_0$. Let $\overline{\scrV|_U}$ be the Hecke transform of $\scrV|_U$ with respect to $\overline{\Lambda}_1 , \ldots , \overline{\Lambda}_t$, which is defined as in \eqref{HeckeTransform2}. Then we have the embedding
\[
\bIQ_e \left( \scrV|_U \right)_\delta \ \hookrightarrow \ \bIQ_e \left( \overline{\scrV|_U} \right)_\epsilon .
\]
for a suitable $\epsilon \in \Z_2$. By openness of the property $\cP$ in families, after shrinking $U$ if necessary we may assume that $\bIQ_e \left( \overline{\scrV|_b} \right)_\epsilon$ has property $\cP$ for all $b \in U$. Then by Proposition \ref{Invariance1}, we see that
\begin{equation} \label{constant}
N_{C,e}^{v+tn} (Q ; \, t;\, \overline{\scrV|_b}) \hbox{ is constant with respect to } b \in U .
\end{equation}
Now let $b'$ be any other point of $B$. Then since any component of $B$ is quasi-projective, we can find a chain of open subsets $U = U_0 , U_1 , \ldots , U_h$ of components of $B$ such that $U_k \cap U_{k+1}$ is nonempty for $0 \le k \le h-1$ and $b' \in U_h$, together with the Hecke transforms $\overline{\scrV|_{U_k}}$ of $\scrV|_{U_k}$ for $1 \le k \le h$ where $\overline{\scrV|_{U_k}}$ has degree $v + t_k n$ for some $t_k \ge 0$, and $\bIQ_e( \overline{\scrV|_b})$ has property $\cP$ for each $b \in U_k$. Take $b \in U_{k-1} \cap U_k$. Then by part (1), we have
\[
N_{\cC_b , e}^{v + t_{k-1} n} ( Q (\alpha);\, t_{k-1} )_{\epsilon_k} \ = \ N_{\cC_b , e}^{v + t_k n}(Q(\alpha);\, t_k)_{\epsilon_k}.\]
for suitable choices of $\epsilon_k \in \Z_2$. By definition of $\tN_{C,e}(Q;\,V)_\delta$ and \eqref{constant} above, $\tN_{\cC_b,e}(Q(\alpha);\, \scrV|_b)_\delta$ is constant with respect to $b \in B$.
\end{proof}
\begin{remark}
For $\bIQeV_\delta$ not having property $\cP$, the numbers
 $\tN_{g,e}^v(Q(\alpha); u ;V)_{\delta}$ could be compared
with virtual intersection numbers on $\bIQ_e(V)_\delta$ once they were properly defined. In fact, Sinha \cite{Si} gave a virtual intersection theory on an isotropic Quot scheme, and then explicitly computed virtual intersection numbers for isotropic subsheaves of rank $2$, which corresponds to  the case $n=2$ in our result. We conjecture that those virtual intersection numbers are equal to the corresponding numbers $\tN_{g,e}^v(Q(\alpha); u ;V)_{\delta}$.

On the other hand, when $\bIQeV_\delta$ has property $\cP,$ we have
\[
\tN_{g,e}^v(Q(\alpha); u ;V)_{\delta} \ = \ N_{g,e}^v(Q(\alpha); u ;V)_{\delta}
\]
by Proposition \ref{Prop:coincide} below. Thus $\tN_{g,e}^v(Q(\alpha); u ;V)_{\delta}$ gives the actual (not virtual) count of points in the intersection \eqref{intersection1}. In particular,  when $\dim \: \bIQeV_\delta=0$ (with $Q(\alpha)=1$),
the number $\tN_{g,e}^v(1 ;V)_{\delta}$ counts maximal subbundles of $V$ under proper conditions.
\end{remark}

\subsection{Relations between intersection numbers}
Now we investigate  relations among the numbers $\tN^v_{C, e} ( Q(\alpha); u ; V )_\delta$.

\begin{proposition} \label{Prop:coincide} Let $V$ be any orthogonal bundle of degree $w$ such that $\bIQeV_\delta$ has property $\cP.$ Then the two definitions of intersection number coincide:
$$ \tN_{C,e}^v ( Q (\alpha); u ; V)_\delta \ = \ N_{C,e}^v ( Q (\alpha) ; u ; V)_\delta. $$

\end{proposition}
\begin{proof} Since $\bIQeV_\delta$ already has property $\cP,$ we may take $\tV=V.$ Then the proposition is immediate.
\end{proof}
The following is a generalization of Corollary \ref{Coro:equality-numbers}.
\begin{proposition}\label{Prop:equality-Hecke}
Let $V$ and $\tV$ be as in \eqref{Hecke-Sequence}, so we have the embedding
$\bIQeV_\delta\hookrightarrow \bIQ_e(\tV)_\epsilon.$ Let $u$ be a nonnegative integer and $Q(\alpha)$  any homogeneous polynomial of degree $I_u(n,\ell,e)$.
Then we have
\begin{equation}
\tN_{C,e}^v(Q(\alpha); u ; V)_\delta=\tN_{C,e}^{v+tn}(Q(\alpha); u+t; \tV)_\epsilon.
\end{equation}
\end{proposition}
\begin{proof}
This is immediate from the definition of $\tN_{C,e}^v(Q(\alpha); u ;V)$.
\end{proof}

Let $V$ be an $L$-valued orthogonal  bundle of degree $v$ over $C$. Let $\widehat{L}$ be a line bundle of degree $\hat{\ell}$ over $C$. Then $V\otimes \widehat{L}$ is an $L \otimes \widehat{L}^2$-valued orthogonal bundle of degree $v + 2n \hat{\ell}$.
Then the association
\[ [ E \to V] \ \mapsto \ \left[ ( E \otimes \widehat{L} ) \to ( V \otimes \widehat{L} ) \right] \]
defines an isomorphism $\bIQeV \isom \bIQ_{e + n\hat{\ell}} ( V \otimes \widehat{L} )$. Thus if $\bIQeV_\delta$ is nonempty,  then there is a unique $\theta\in \Z_2$ such that  $\bIQeV_\delta\isom  \bIQ_{e + n\hat{\ell}} ( V \otimes \widehat{L} )_{\theta}$. Thus we have
\begin{proposition} \label{Prop:tranform-GW} Let $V$ and $\widehat{L}$ be as above. Let $u$ be a nonnegative integer and $Q(\alpha)$ be homogeneous polynomial of degree $I_u(n,\ell,e).$ If  $\bIQeV_\delta$ is nonempty,  then there is a unique $\theta\in \Z_2$ for which we have the equality
$$ \tN_{C,e}^{v} ( Q (\alpha) ; u; V)_{\delta} \ = \ \tN_{C,e + n \hat{\ell}}^{v+ 2n\hat{\ell}}( Q(\alpha) ; u ; V \otimes \widehat{L})_{\theta} . $$
\end{proposition}

\begin{proposition} \label{Relation of GW-invariants} Let $V$ be a $L$-valued orthogonal bundle of degree $v=n\ell$ on $C$. Assume $\bIQeV_\delta$ is nonempty. Let $u$ be a nonnegative integer and $Q(\alpha)$ a homogeneous polynomial of degree $I_u(n,\ell,e).$
Then \begin{enumerate} \item
for each $k\geq 0,$  we have \be \tN_{C,e}^v(Q(\alpha); u ;V)_{\delta}=\tN_{C,e-2kn}^v(Q(\alpha); u+4k ;V)_{\delta},\ee
\item for any two nonnegative integers $r,s$, there is an orthogonal bundle $W$ of degree $w:=v+rn-2sn$ such that  $\bIQ_{e-sn}(W)_{\theta}$ is nonempty for some $\theta \in \Z_2$, and
\be \tN_{g,e}^v(Q(\alpha); u; V)_{\delta}=\tN_{g,e-sn}^w(Q(\alpha); u+r;\, W)_{\theta}.\ee

\end{enumerate}
\end{proposition}
\begin{proof}
The proof of $(1)$ is essentially the same as the proof of   the symplectic case in \cite[Proposition 5.13]{CCH1}.
For a proof of $(2)$,
let $\tV$ be the Hecke transform of $V$ taken at distinct $r$ points, $p_1,\ldots ,p_r$, so that there is an embedding of schemes $\bIQeV_\delta\hookrightarrow \bIQetV_{\epsilon}$ for some $\epsilon \in \Z_2$.  Choose a line bundle $M$ of degree $-s$ and set $W:=\tV\otimes M,$ so we have $\bIQ_e(\tV)_\delta\isom \bIQ_{e-ns}(W)_\theta$ for some $\theta\in \Z_2.$ Then the  proposition is immediate from Propositions \ref{Prop:equality-Hecke} and \ref{Prop:tranform-GW}.
\end{proof}

\begin{corollary}\label{Two relations}
Let $V$ be an orthogonal bundle of degree $v=n\ell.$  Assume that $\bIQeV_\delta$ is nonempty.  Then the following hold.
\begin{enumerate}
 \item \label{relation1} If $\ell=4m-a$ for some integers $m, a\in \Z$ with $0\leq a\leq 3,$ then there is an orthogonal bundle $W_1$ of degree $0$ and $\epsilon_1\in \Z_2$ for which
\begin{equation}\label{equality1}\tN_{g,e}^v(Q(\alpha); u ;V)_{\delta} \ = \ \tN_{g,e-2mn}^0(Q(\alpha); u+a;W_1 )_{\epsilon_1}.\end{equation}
\item \label{relation2} If $\ell=4k+2-b$ for some $k, b\in \Z$ with $0\leq b \leq 3$, then there is an orthogonal bundle $W_2$ of degree $0$ and $\epsilon_2\in \Z_2$ for which
\begin{equation}\label{equality2}
\tN_{g,e}^v(Q(\alpha); u ; V)_{\delta} \ = \ \tN_{g,e-(2k+1)n}^0(Q(\alpha); u+b ; W_2 )_{\epsilon_2}.\end{equation}
\end{enumerate}
\end{corollary}
\begin{proof} We apply $(2)$ of Proposition \ref{Relation of GW-invariants} to both cases.
To show (\ref{relation1}), take $r=a$ and $s=2m$ and to show (\ref{relation2}), take $r=b$ and $s=2k+1.$
\end{proof}
Note that  $\ell$ can be expressed in two ways as in  (1) and (2) above, and so the intersection number $\tN_{g,e}^v(Q(\alpha);V)_{\delta}$ has two different expressions
\eqref{equality1} and \eqref{equality2}.

\subsection{Counting isotropic subbundles}

Recall that the parameter space $\OG(n)$ of isotropic subspaces in the vector space $\bC^{2n}$ equipped with a symmetric form has two isomorphic components $\OG(n)_0$ and $\OG(n)_1$.
Let $V$ be a trivial orthogonal bundle  $\cO_C^{\oplus 2n}$ over $C$. Then $\OG(V)=\OG(V)_0\cup \OG(V)_1$ where $\OG(V)_i= C\times\OG(n)_i$ for $i=0,1.$ Then $\bIQeoV$ is nonempty if and only if $e$ is  even,
since $w_2(V)= 0\ \mathrm{mod} \ 2$.  On the other hand, a morphism $f:C\to \OG(n)_\delta$  of degree $d$ defines a section $\phi:C\to \OG(V)_\delta$, corresponding to an isotrpoic subbundle $E$ of $V$ with degree $-2d$   (\cite{KT1}). Hence there is an isomorphism $$\bIQeoV_\delta\isom \mathrm{Mor}_{-\frac{e}{2}}(C,\OG(n)_\delta).$$ Thus $\bIQeV_\delta$ is a compatification of $\mathrm{Mor}_{-\frac{e}{2}}(C,\OG(n)_\delta)$, alternative to the Kontsevich's moduli space of stable maps to $\OG(n)_\delta.$ In fact, Kresch and Tamvakis used  $\bIQeV$ (precisely, $\IQ_e(V)$) to compute the quantum cohomology ring of $\OG(n)_0$   for $C=\mathbb{P}^1$; see \cite{KT1}. For the trivial bundle $\cO_C^{\oplus 2n}$ on $C$, we have the following result.

\begin{corollary} \label{Prop:trivial} Let $C$ be a smooth projective curve of genus $g$. Let $u$ be a nonnegative integer and $\lambda^{1} , \ldots , \lambda^{m} \in \cD(n-1)$ strict partitions such that $\sum_{i=1}^m|\lambda^i| = I_u(n,0,e)$. Set $Q(\alpha) = \prod_{i=1}^m \Pt_{\lambda^i}(\alpha)$. Then we have the equality for both $\delta\in \{0, 1\}$,
 \begin{equation} \label{multiplication}
\tN^{0}_{C,e}(Q(\alpha); u ; \cO_C^{\oplus 2n})_\delta \ = \ \left\langle \tau^u_{\rho_{n-1}}, \tau_{\lambda^1} , \ldots\hspace{0.02in} , \tau_{\lambda^m}\right\rangle_{g,\frac{|e|}{2}}.
\end{equation}
\end{corollary}

\begin{proof}
Consider the monomial $P(\alpha) = \prod_{i=1}^s \alpha_{k_i}$ of (weighted) degree $I_u(n,0,e)$.  Using Proposition \ref{general intersection} and the argument in the proof of Corollary \ref{Coro:equality-numbers}, we obtain
\be
\label{eq5} \tN_{C,e}^0 ( P(\alpha) ; u ; \cO_{C}^{\oplus 2n}  ) \ = \ \left\langle \tau^u_{\rho_{n-1}}, \tau_{k_1}, \ldots ,\tau_{k_s} \right\rangle_{g,\frac{|e|}{2}} .
\ee
On the other hand, in view of the Vafa--Intriligator-type formula, the Gromov--Witten invariant $\left\langle \tau_{\lambda^1}, \ldots ,\tau_{\lambda^m}\right\rangle_{g,d}$ only depends on the product $\prod_{i=1}^m \tau_{\lambda^i}$ of the arguments. Thus we have
\[\left\langle \tau^u_{\rho_{n-1}}, \tau_{\lambda^1}, \ldots ,\tau_{\lambda^m}\right\rangle_{g,d} \ = \ \left\langle \tau^u_{\rho_{n-1}}\cdot \prod_{i=1}^m\tau_{\lambda^i}\right\rangle_{g,d}.\]
Recall that the Chow ring  $\mathrm{CH}^*(\OG(n)_0)$ is generated by $\tau_1,\ldots,\tau_{n-1}.$ In particular, we have  \[\tau^u_{\rho_{n-1}}\cdot \prod_{i=1}^m\tau_{\lambda^i} \ = \ \Pt_{\rho_{n-1}}^u (\tau)\cdot Q(\tau),\] where
$\Pt_{\rho_{n-1}}^u (\tau)\cdot Q(\tau)$ is  the polynomial $\Pt_{\rho_{n-1}}^u (\alpha)\cdot Q(\alpha)$ with $\alpha$ replaced by $\tau=(\tau_1,\ldots, \tau_{n-1}).$ Then the equality  (\ref{multiplication}) follows from the linearity of both sides \eqref{eq5}.
\end{proof}

Corollary \ref{Prop:trivial} together with (1) of Propositions \ref{Relation of GW-invariants} yields the following recursive relation among Gromov-Witten invariants of $\OG(n)_\delta$ for any $\delta \in \Z_2$.

\begin{corollary} Let $C$ be a smooth projective curve of genus $g$. Suppose $\sum_{i=1}^m|\lambda^{i}|= 2(n-1)d + \frac{n(n-1)}{2}(1-g)$ for $d\geq 0.$ Then for any integer $s\geq 0$, we have
$$
\left\langle \tau_{\lambda^1} , \ldots \hspace{0.02in} , \tau_{\lambda^m} \right\rangle_{g,d} \ = \ \left\langle \tau_{\rho_{n-1}}^{4s} , \tau_{\lambda^1} , \ldots \hspace{0.02in} , \tau_{\lambda^m} \right\rangle_{g,d+sn} .
$$
\end{corollary}

From now on we restrict ourselves to the intersection numbers $\tN_{g,e}^v(Q(\alpha); V)$ (the case $u=0$).
Let $L$ be a line bundle of degree $\ell.$ Let  $m, a, k, b\in \Z$ be integers such that $\ell=4m-a$ for $0\leq a\leq 3,$ and $\ell=4k+2-b$ for $0\leq b\leq 3.$
\begin{corollary}\label{Last Coro} Let $V$ be an orthogonal bundle of degree $v=n\ell.$ Let $Q(\alpha)$ be a homogeneous polynomial of degree $I(n, \ell, e).$ Assume $\bIQeV_\delta$ is nonempty.  Then the intersection number $\tN_{g,e}^v(Q(\alpha);V)_\delta$ is given as follows.
\begin{enumerate}
\item  If $e$ is even, then (for any $n$)
\begin{equation}\label{formula for even}
 \tN_{g,e}^v(Q(\alpha);V)_\delta \ = \ \frac{1}{2^{e-2mn}}\sum_{J\in
\cI_{n-1}}S_{\rho_{n-1}}(\zeta^J)^{g-1}\Pt_{\rho_{n-1}}(\zeta^J)^aQ(\zeta^J),\end{equation}

\item if both $n$ and $e$ are odd, then \begin{equation}
 \tN_{g,e}^v(Q(\alpha);V)_\delta \ = \ \frac{1}{2^{e-(2k+1)n}}\sum_{J\in
\cI_{n-1}}S_{\rho_{n-1}}(\zeta^J)^{g-1}\Pt_{\rho_{n-1}}(\zeta^J)^bQ(\zeta^J).\end{equation}

\end{enumerate}
\end{corollary}
\begin{proof}Recall that for an even integer $f\leq 0$,  we have
\[\tN_{g,f}^0(Q(\alpha);\cO_C^{\oplus 2n})_{0} \ = \ \tN_{g,f}^0(Q(\alpha);\cO_C^{\oplus 2n})_{1}.\]

The formulas $(1)$ and $(2)$ follow from  Corollaries  \ref{Two relations}, \ref{Prop:trivial} and  Proposition \ref{corollary:gro;oge}.
\end{proof}

\begin{remark}\label{Last remark}  We notice that if $\ell$ is even, then for each even  $e$ with $e\equiv w_2(V) \ \mathrm{mod}\  2$ (so that $\bIQeV$ is nonempty),  the formula in \eqref{formula for even} does not depend on $\delta \in \{0,1\},$ i.e., we have $ \tN_{g,e}^v(Q(\alpha);V)_0= \tN_{g,e}^v(Q(\alpha);V)_1$. This confirms the speculation in \cite[Conjecture 8.2]{CH4}, and   it can also be shown as in the proof of Proposition \ref{counting odd rank}.

\end{remark}

Note that  we do not have a closed formula for $\tN_{g,e}^v(Q(\alpha);V)_\delta$ when $n$ is even and $e$ is odd. This is because the bundles in this case cannot be connected  to the trivial bundle by Hecke transforms and tensoring by line bundles. Recall that intersection numbers for the trivial bundle are explicitly computed using Gromov-Witten invarants.  On the other hand, an orthogonal bundles $V$ with $n$ even and $e$ odd can be connected to  an  orthogonal bundles $W$ of degree $0$ and with $w_2(W)=1$. Thus it would be interesting to compute intersection numbers for an orthogonal bundle $W$ of degree $0$ and with $w_2(W)=1$.

\section{Counting formulas for orthogonal bundles}\label{Section:counting formulas}
In this section, we give counting formulas for maximal isotropic subbundles of orthogonal bundles.
Let $V$ be a general stable $L$-valued orthogonal  bundle over $C$ of rank $r=2n$, or $2n+1$ such that $\bIQ_{e_0}(V)$ is a nomempty smooth scheme of zero dimension.(See Theorem \ref{e_0maximal} for a condition on $e_0$.) Then the intersection number corresponding to $Q(\alpha)=1$:
  \[\int_{\bIQ_{e_0}(V)}1\]
  is precisely the number of maximal isotropic subbundles of $V$. Recall that this number is invariant under a deformation of $V$ and $C$, and we denote it by $N(g,r,\ell,e_0)$.

\subsection{Even rank case}  \label{ResultsEven}

Assume  $r=2n.$ If $\ell$ is odd,  $N(g,2n,\ell,e_0)=\tN_{g,e}^v(Q(\alpha);V)_\delta$ for exactly one   $\delta \in \mathbb{Z}_2$. If $\ell$ is even,
$N(g,2n,\ell,e_0)= 2\tN_{g,e}^v(Q(\alpha);V)_\delta$ by Remark \ref{Last remark}. Thus from Corollary \ref{Last Coro}, we have
\begin{corollary} \label{counting formula} \quad
 Let $e_0 := -\frac{1}{2} n(g-1 - \ell)$. Then the number $N(g,2n,\ell,e_0)$ is given as follows.
\begin{enumerate}
\item Assume that $e_0$ is even.  If $\ell=4m-a$ for some $m, a\in \Z$ and $0\leq a\leq 3,$  then   we have
\begin{displaymath} N (g, 2n, \ell, e_0) \ = \
\begin{cases}

\scriptstyle{\scriptstyle{{A_1}\sum_{J\in
\cI_{n-1}}S_{\rho_{n-1}}(\zeta^J)^{g-1}\Pt_{\rho_{n-1}}(\zeta^J)^a}}, & \text{if} \ \: a=0,2 , \\

\scriptstyle{\scriptstyle{{A_2}\sum_{J\in
\cI_{n-1}}S_{\rho_{n-1}}(\zeta^J)^{g-1}\Pt_{\rho_{n-1}}(\zeta^J)^a}}, & \text{if} \ \: a=1,3 ,\\
\end{cases}
\end{displaymath}
where $A_1=2^{2mn-e_0+1}=\sqrt{2}^{n(a+g-1)+2}$ and $A_2=2^{2mn-e_0}=\sqrt{2}^{n(a+g-1)}$.
\item Assume that  $n$ is odd and $e_0$ is odd. If $\ell=4k+2-b$ for $0\leq b\leq 3$, then
\begin{displaymath} N (g, 2n, \ell, e_0) \ = \
\begin{cases}

\scriptstyle{\scriptstyle{{B_1}\sum_{J\in
\cI_{n-1}}S_{\rho_{n-1}}(\zeta^J)^{g-1}\Pt_{\rho_{n-1}}(\zeta^J)^b}}, & \text{if} \ \: b=0,2 , \\

\scriptstyle{\scriptstyle{{B_2}\sum_{J\in
\cI_{n-1}}S_{\rho_{n-1}}(\zeta^J)^{g-1}\Pt_{\rho_{n-1}}(\zeta^J)^b}}, & \text{if} \ \: b=1,3 ,\\
\end{cases}
\end{displaymath} where $B_1=2^{n(1+2k)-e_0+1}=\sqrt{2}^{n(b+g-1)+2}$ and  $B_2=2^{n(1+2k)-e_0}=\sqrt{2}^{n(b+g-1)}.$
\end{enumerate}
\end{corollary}

\subsection{The odd rank case} \label{ResultsOdd}

We shall now use a technique applied in \cite{CCH2}, \cite{CH4} and elsewhere to deduce statements of enumerative geometry for isotropic Quot schemes of odd rank orthogonal bundles from those established for bundles of even rank.
 Note that if  $\bIQ_{e_0}(V)$ is a nomempty smooth scheme of dimension zero, we have $\mathrm{IQ}_{e_0}(V)=\bIQ_{e_0}(V)=\bIQ_{e_0}^\circ (V)$.

Let $V$ be an orthogonal bundle of rank $2n+1$ for $n \ge 1$. As observed in \cite[Lemma 2.4]{CCH2}, there is a uniquely determined line bundle $M$  such that $ M^2 \cong L$ and $\det V \cong M^{2n+1}$. Let $m = \deg (M)$. The following confirms and generalizes \cite[Conjecture 8.1]{CH4}.
%
\begin{proposition} \label{counting odd rank}
For $\ell= 2m$, let $e_0 := -\frac{1}{2} (n+1)(g-1) +mn$. Then we have

\[N ( g, 2n+1, \ell , e_0 ) \ = \ \frac{1}{2} \cdot N ( g, 2n+2, \ell, e_0+m ) .\]
\end{proposition}
\begin{proof}
 By \cite[Lemma 8.1]{CCH2}, we have an isomorphism
\[
\bIQo_e(V) \ \cong \ \bIQo_{e+m} (V \perp M)_\delta
\]
for both $\delta \in \{0,1 \}$. Note that $V \perp M$ is an $L$-valued orthogonal bundle of rank $2n+2$.  For $e=e_0$ and for a general $V$, we have seen in Theorem \ref{e_0maximal} that  $\bIQo_{e}(V)$ is a smooth scheme of dimension zero.
By the invariance of the intersection numbers under a deformation which was discussed in \S \: \ref{section:invariance}, the length of $\bIQo_{e+m} (V \perp M)_\delta$ is equal to that of $\bIQ_{e+m} (W)_\delta$ for a general $W$ in the same component as $V \perp M$. Since $\bIQo_{e+m} (V \perp M)$ is isomorphic to the union of two copies of $\bIQ_e(V) =\bIQo_e(V)$,  the statement follows.
\end{proof}

 \begin{remark}\label{Remark:Not have}
  In Corollary \ref{counting formula} and Proposition \ref{counting odd rank} we don't have a closed formula for $N(g, 2n, \ell, e_0)$  when $n$ is even and $e_0$ is odd, and accordingly a formula  $N(g, 2n+1, \ell, e_0)$ when both $n$  and  $e_0+m$ are odd. See the last paragraph of  \S\:\ref{section:GW numbers} for more details.
 By the way, in \cite[Theorem 5.16]{CH4},   explicit formulas are obtained for low rank cases, and in particular for any $g$, it is shown that
  \[ N(g, 4,0,1-g)=2\cdot2^g ; \ \ \  \mathrm{and} \ \ \  N(g, 3, 0, 1-g) =  2^g .\]
  \end{remark}

\subsection{Counting results for orthogonal bundles of low rank} \label{low_rank}

Now we list some explicit counting results obtained from the previous formulas for orthogonal bundles of rank $r$, where $3 \le r \le 6$.
\begin{theorem} \label{main1.1}
Let $C$ be any curve of genus $g \ge 2$. Let $w_2 \in \{0,1\}$ be the 2nd Stiefel--Whitney class.
\begin{enumerate}
\item If $g$ is odd and $w_2 = 0$, then $N ( g , 4 , 0 , 1-g ) = 2^{g+1}$.
\item If $g$ is odd and $w_2 = 0$, then $N ( g , 3 , 0 , 1-g ) \ = 2^g$.
\item If either  ($g \equiv 1 \mod 4$, $w_2 = 0$) or ($g \equiv 3 \mod 4$, $w_2 = 1$) then
\[ N \left( g , 6 , 0 , \frac{3}{2}(1 - g) \right) \ = \ 2^{2g + 1} . \]
\item If $g \equiv 2 \mod 4$ and $w_2=0$, then $N \left( g , 6 , 1 , \frac{3}{2} ( 2 - g ) \right) = 2^{2g}$.
\item If either ($g \equiv 1 \mod 4$, $w_2 = 0$) or ($g \equiv 3 \mod 4$, $w_2 = 1$), then
\[ N \left( g , 5 , 0 , \frac{3}{2}(1 - g) \right) \ = \ 2^{2g} . \]
\end{enumerate}
\end{theorem}

\begin{proof}
By direct computation using Corollary \ref{counting formula}, parts (1), (3) and (4)  are obtained. Using Proposition \ref{counting odd rank}, parts (2) and (5)  are  consequences of (1) and (3) respectively.  Note that when $\ell=0$, the parity of the degree of a rank $n $ isotropic subbundle of $V$ coincides with $w_2 (V)$.
\end{proof}

Note that this result coincides with that of \cite{CH4}, where the numbers were computed by using canonical isomorphisms for low rank orthogonal bundles. In particular,  (5) confirms \cite[Conjecture 7.8]{CH4}.\\

\end{document}